\documentclass[11pt]{amsart}

\usepackage[margin=1 in]{geometry}	

\usepackage{amssymb}	
\usepackage{amsmath}	
\usepackage{amsthm}		
\usepackage{paralist}		
\usepackage{color}			

\usepackage{hyperref}	
\hypersetup{colorlinks, linkcolor=blue}
\usepackage{soul}
\soulregister\cite7
\soulregister\eqref7
\soulregister\eqref7
\hypersetup{colorlinks, linkcolor=blue}

\newtheorem{thm}{Theorem}[section]
\newtheorem{lem}[thm]{Lemma}
\newtheorem{pro}[thm]{Proposition}
\newtheorem{cor}[thm]{Corollary}
\newtheorem*{rem}{Remark}

\theoremstyle{definition}
\newtheorem{definition}[thm]{Definition}

\newcommand\ddd{\mathrm{d}}

\newcommand\supp{\mathrm{supp}}

\newcommand\bR{\mathbb{R}}

\newcommand\bN{\mathbb{N}}
\newcommand\bZ{\mathbb{Z}}

\def \l {\left}
\def \r {\right}

\makeatletter
\@namedef{subjclassname@2020}{
	\textup{2020} Mathematics Subject Classification}
\makeatother

\allowdisplaybreaks[4]

\begin{document}
\title[On $L^p$ extremals for fractional surface]{On $L^p$ extremals for Fourier extension estimate to fractional surface}
\author{Boning Di}
\author{Ning Liu}
\author{Dunyan Yan}

\date{}
\thanks{This work was supported in part by the National Key R\&D Program of China [Grant No. 2023YFC3007303], the National Natural Science Foundation of China [Grant Nos. 12071052 \& 12271501] and the China Postdoctoral Science Foundation [Grant Nos. GZB20230812 \& 2024M753436].}

\begin{abstract}
	This article investigates the Fourier extension operator associated with the fractional surface $(\xi,|\xi|^{\alpha})$ for $\alpha\geq 2$. We show that the relevant $L^p\to L^q$ Fourier extension inequality possesses extremals for all exponents $p\in[1,2]$. Moreover, for all $p\in(1,2]$, the corresponding $L^p$-extremal sequences are precompact up to symmetries.
\end{abstract}
\maketitle

\section{Introduction}
For $\alpha\geq 2$, we investigate the fractional surface $\mathbb{P}^d_{\alpha}=\{(\xi,|\xi|^{\alpha}): \xi \in\mathbb{R}^d\}$ in $\bR^{d+1}$, equipped with the following projection measure
\[\int_{\mathbb{R}^{d+1}} g(\xi, t) \mathrm{d}\sigma_{\mathbb{P}^d_{\alpha}}:=\int_{\mathbb{R}^d}g(\xi,|\xi|^{\alpha})\mathrm{d}\xi.\]
This serves as a natural generalization of the paraboloid, while one should notice that this fractional surface can have zero Gaussian curvature at the origin. The associated Fourier extension operator is given by
$$\mathcal{E}_{\alpha}f(x,t):=\int_{\mathbb{R}^d} e ^{i(x,t)(\xi,|\xi|^{\alpha})} f(\xi) \mathrm{d}\xi,\quad (x,t)\in\mathbb{R}^{d+1}.$$
In this setting, the Fourier restriction theory investigates whether $\mathcal{E}_{\alpha}$ is a bounded operator from $L^p(\bR^d)$ to $L^q(\bR^{d+1})$ for some pair $(d,p,q)$. Indeed, applying some rescaling and stationary phase arguments \cite[Chapter 8, page 363, Exercise 5.13]{Stein1993}, one concludes that the desired Fourier extension estimate
\begin{equation}\label{neq1}
	\Vert\mathcal{E}_{\alpha}f\Vert_{L^q(\mathbb{R}^{d+1})}\leq C_{\alpha,p}\Vert f\Vert_{L^p(\mathbb{R}^{d})}
\end{equation}
can only hold if the pair $(d, p,q)$ satisfies the following necessary condition
\begin{equation} \label{E:Necessary condition}
	q=\frac{d+\alpha}{d}p', \quad q>p.
\end{equation}

The typical paraboloid case $\alpha=2$ is the well-known Fourier restriction conjecture for the paraboloid. Up to now, this conjecture has been solved for dimension $d = 1$ and remains open
 for higher dimensions. For example, when $d=2$, this conjecture has been confirmed for $q>3+\frac{3}{14}$ by Wang-Wu \cite{WW2022}. For more related topics, one can refer to publications like \cite{Demeter2020,Tao2004}.
 
 For general fractional surfaces with $\alpha \geq 2$, the relevant Fourier restriction phenomena have also been studied by several mathematicians, see for instance the papers \cite{CKZ2013,KPV1991,Stovall2017}. Indeed, the classical work of Kenig-Ponce-Vega \cite{KPV1991} states that $\mathcal{E}_{\alpha}$ is a bounded operator from $L^2(\bR^d)$ to $L^{\frac{2d +2\alpha}{d}}(\bR^{d+1})$. Since the operator $\mathcal{E}_{\alpha}$ is obviously $L^1(\bR^d)$ to $L^{\infty}(\bR^{d+1})$ bounded, by interpolation, one can see that $\mathcal{E}_{\alpha}$ is bounded from $L^p(\mathbb{R}^d)$ to $L^{q}(\mathbb{R}^{d+1})$ for all $p\in[1,2]$ with $q=\frac{d+\alpha}{d}p'$.

Besides the boundedness of Fourier extension operators, the extremal functions and norms of these Fourier extension operators have received substantial attention recently. In other words, one could be interested in the sharp form of inequality \eqref{neq1} and the optimal constant 
$$M_{\alpha,p}:=\sup_{\Vert f\Vert_{L^p(\mathbb{R}^d)}=1}\Vert\mathcal{E}_{\alpha}f\Vert_{L^q(\mathbb{R}^{d+1})}.$$
A sequence of functions $\{f_n\}\subset L^p(\mathbb{R}^d)$ is called an \textit{extremal sequence} for $M_{\alpha,p}$ if it satisfies
$$\Vert f_n\Vert_{L^p(\mathbb{R}^d)}=1,\quad \lim_{n\to\infty}\Vert\mathcal{E}_{\alpha}f_n\Vert_{L^q(\mathbb{R}^{d+1})}=M_{\alpha,p}.$$
And a function $f\in L^p(\mathbb{R}^d)$ is called an \textit{extremal} for $M_{\alpha,p}$ if it satisfies
$$\Vert f\Vert_{L^p(\mathbb{R}^d)}=1,\quad \Vert\mathcal{E}_{\alpha}f\Vert_{L^q(\mathbb{R}^{d+1})}=M_{\alpha,p}.$$
For convenience, we introduce some further terminologies. First, the relevant \textit{symmetries} for the inequality \eqref{neq1} are space-time translation and scaling, defined as follows
$$\mathcal{G}(h_0,x_0,t_0)f(\xi):=h_0^{-\frac{d}{p}} e ^{- i (x_0,t_0)(\xi,|\xi|^{\alpha})}f\left(\frac{\xi}{h_0}\right),\quad (h_0,x_0,t_0)\in\mathbb{R}_+\times\mathbb{R}^{d}\times\mathbb{R}.$$
Notice that these symmetries $\mathcal{G}(h_0,x_0,t_0)$ preserve the norm in $L^p(\mathbb{R}^d)$ and
$$\Vert\mathcal{E}_{\alpha} \mathcal{G} (h_0,x_0,t_0) f\Vert_{L^q(\mathbb{R}^{d+1})}=\Vert\mathcal{E}_{\alpha}f\Vert_{L^q(\mathbb{R}^{d+1})}.$$
Denote the group generated by these symmetries as $G$. Then a sequence of functions $\{f_n\}\subset L^p(\mathbb{R}^d)$ is called \textit{precompact up to symmetries} if there exists a sequence of symmetries $\{\mathcal{G}_n\}\subset G$ such that $\{\mathcal{G}_n f_n\}$ has a convergent subsequence in $L^p(\mathbb{R}^d)$. For an extremal sequence, notice that the precompactness up to symmetries implies the existence of the extremal function.

For the case $p=2$, the extremal problems for Fourier extension operators have been well-studied. There are many approaches involved in these works such as concentration-compactness principle/missing mass method/profile decomposition argument \cite{BOQ2020,CS2012A&P,DY2023,FLS2016,FS2018,JSS2017,Kunze2003,Ramos2012,Shao2009EJDE,Shao2016}, Euler-Lagrange type equation \cite{Carneiro2009,CS2012Adv,Foschi2007,Foschi2015,HZ2006,JS2016}, geometric comparison principle \cite{BOQ2020,OQ2018,OQ2020}, bilinear restriction theory \cite{COSS2021,DY2024,FLS2016,JSS2017}. Indeed, the first and the third authors have summarized some of the relevant works in \cite{DY2023,DY2024}, and the reader can also refer to the survey papers \cite{FO2017,NOT2023,Oliveira2024} as well as the references therein for further details on recent progress.

The first $L^p$-based result is due to Christ-Quilodr\'an \cite{CQ2014} in 2014, who showed that Gaussian functions are not extremals when $p\neq \{1,2\}$, by investigating the associated Euler-Lagrange equation; then in 2020, Stovall \cite{Stovall2020} proved the precompactness of $L^p$-extremal sequences conditional on improvements toward the restriction conjecture, by generalizing the $L^2$-based profile decomposition argument in \cite{Shao2009EJDE}; after that, the strategy given by Stovall is also used in investigating other surfaces or curves such as the moment curve \cite{BS2023}, the sphere \cite{FS2024}, the monomial curves \cite{BS2024} and the cone \cite{NOST2023}.

\textbf{In this article, we investigate the $L^p$-extremal theory for Fourier extension operators, where the associated surfaces (fractional surface) can have vanishing Gaussian curvature at some point}. This vanishing-curvature property prevents us from directly applying Tao's bilinear restriction estimate \cite{Tao2003}, which is a crucial tool for establishing the desired profile decomposition consequences. For the fractional surface, we obtain the following result.
\begin{thm}\label{n1}
	Let $d\geq 2$ and $\alpha> 2$. Then for all $p\in[1,2]$, the extremals for $M_{\alpha,p}$ exist; moreover, for all $p\in(1,2]$, the extremal sequences for $M_{\alpha,p}$ are precompact up to symmetries.
\end{thm}

\begin{rem}
    Recall that Kenig-Ponce-Vega's result \cite[Theorem 2.1]{KPV1991} guarantees the boundedness of $M_{\alpha,p}$ for $p\in[1,2]$. In addition, if one can confirm that $M_{\alpha,p}$ is finite for $p\in[1,p_0]$ with some $p_0>2$, then our method will directly give the existence of extremals for $p\in[1,p_0)$.
\end{rem}

To prove Theorem \ref{n1}, we follow the outline given by Stovall \cite{Stovall2020}. Firstly, one should establish some refined estimates to demonstrate that extremal sequences possess good frequency localization. Subsequently, the functions become bounded with compact support, which can be handled by employing the $L^2$-based profile decomposition. Here for establishing the $L^2$-based profile decomposition, one technical problem arises from the fact that the fractional surface's Gaussian curvature vanishes at the origin. Fortunately, Di-Yan \cite[Section 2]{DY2024} has addressed this problem: they divide the fractional surface into countable annuli, and then further divide each annuli into finite regions by dividing the angles. However, that paper \cite{DY2024} used a missing mass method rather than a profile decomposition argument. In this paper, to deal with the $L^p$ situation and follow Stovall's outline, we will establish the full profile decomposition. Here we should remark that the profile decomposition consequence also has many applications in PDE, see for instance \cite{BG1999,BV2007,CK2007,Keraani2001,Keraani2006,MV1998}.

This paper is organized as follows. In Section \ref{S:Preliminaries}, we list some preliminaries, which include the $L^2$-consequence for fractional surface and $L^p$-consequence used in investigating other surfaces. In Section \ref{ns1}, we establish the $L^2$-based profile decomposition and give some corollaries of this profile decomposition. In Section \ref{ns2}, we establish the frequency localization together with the associated $L^p$-based profile decomposition, and then prove our main result Theorem \ref{n1}.

We end this section with some notations. We use $X\lesssim Y$ to denote the estimate $|X|\leq CY$ for some constant $0<C<\infty$, which will not depend on the functions, and sometimes the dependence on $d$ will be shown as $X\lesssim_{d} Y$; then similarly for the notations $X \gtrsim Y$ and $X\sim Y$. The indicator function of a set $E$ will be denoted by $\chi_E$, and we further define $f_E:=\chi_Ef$.

\section{Preliminaries} \label{S:Preliminaries}
In this section, we recall some auxiliary consequences which will be used in our arguments. Some relevant $L^2$-based consequences are shown in Subsection \ref{SubS:L2 preliminaries}, while the $L^p$-based consequences are shown in Subsection \ref{SubS:Lp preliminaries}.

\subsection{The $L^2$ consequences} \label{SubS:L2 preliminaries}
The special $L^2$ case for fractional surface has been studied by the first and the third authors in \cite{DY2024}. Here we rewrite some of their results in the format applicable to this article.

First, we introduce a regular-type Fourier extension operator on the fractional surface and some basic dyadic terminologies. We recall the following operator
\[\widetilde{\mathcal{E}}_{\alpha}f(x,t) :=\int_{\mathbb{R}^d}|\xi|^{\frac{(\alpha-2)d}{2(d+2)}} e ^{ix\xi+it|\xi|^{\alpha}} f(\xi) \mathrm{d}\xi,\]
and the associated Fourier extension type estimate
\begin{equation}\label{neq17}
	\left\Vert\widetilde{\mathcal{E}}_{\alpha}f\right\Vert_{L^{\frac{2(d+2)}{d}}(\mathbb{R}^{d+1})}\leq \widetilde{M}_{\alpha}\Vert f\Vert_{L^2(\mathbb{R}^{d})},
\end{equation}
where 
$$\widetilde{M}_{\alpha}:=\sup_{\Vert f\Vert_{L^2(\mathbb{R}^d)}=1}\Vert\widetilde{\mathcal{E}}_{\alpha} f\Vert_{L^{\frac{2(d+2)}{d}}(\mathbb{R}^{d+1})}.$$
A dyadic interval is an interval of the form $[m2^{-n},(m+1)2^{-n}]$ with $m,n\in\mathbb{Z}$, and a dyadic cube is a product of dyadic intervals all having the same length; the set of all dyadic cubes of side length $2^{-k}$ is denoted by $\mathcal{D}_k$, and the set of all dyadic cubes is denoted by $\mathcal{D}$; for two dyadic cubes $\tau$ and $\tau'$, the notation $\tau\sim\tau'$ means that they have the same length $\ell(\tau) =\ell(\tau')$ but their distance is like $C_{d,\alpha} \ell(\tau)$ with the constant $C_{d,\alpha}$ sufficiently large\footnote{Further details can be seen in \cite[Definition 2.4]{DY2024}}.

The well-known bilinear restriction estimate \cite[Section 9, third remark]{Tao2003} is also used in \cite{DY2024}. However, the Gaussian curvature of the fractional surface vanishes at the origin. To resolve this difficulty, one needs to divide $\mathbb{R}^d$ into dyadic cube annuli as follows
\[\mathcal{A}_N:=\left\{\xi\in\mathbb{R}^d:2^N\leq \Vert\xi\Vert_{\max}\leq 2^{N+1}\right\}, \quad \Vert\xi\Vert_{\max}:=\max\{|\xi_1|,|\xi_2|,\cdots,|\xi_d|\}, \quad N\in\mathbb{Z}.\]
To utilize the Whitney decomposition, one needs to further divide the dyadic cube annuli into several parts based on the angles\footnote{Here the constant $K_{d,\alpha}$ is pretty large and further details can be seen in \cite[Page 7]{DY2024}}
\[\mathcal{A}_N=\bigcup_{j=1}^{K_{d,\alpha}}\mathcal{A}_N^j,\]
and this partitioning directly gives the following fact
\begin{equation}\label{eq17}
	\Vert\mathcal{E}_{\alpha}f_{\mathcal{A}_0}\Vert_{L^p(\mathbb{R}^{d+1})}\sim_{d,\alpha} \max_{1\leq j\leq K_{d,\alpha}} \Vert\mathcal{E}_{\alpha}f_{\mathcal{A}_0^j}\Vert_{L^p(\mathbb{R}^{d+1})}.
\end{equation}
Then the annular-restricted notation $\tau\sim\tau'\subset\mathcal{A}_N^j$ indicates 
\[\tau\sim\tau'\subset\mathcal{A}_N,\quad \tau\cap\mathcal{A}_N^j\ne\varnothing,\quad \tau'\cap\mathcal{A}_N^j\ne\varnothing.\]
With these notations in place, we have the following consequences. 
\begin{lem}[Lemma 2.5 of \cite{DY2024}]\label{12}
	We can divide $\mathcal{A}_0$ into $K_{d,\alpha}$ parts as shown above, such that for each $j\in\{1,2,\cdots,K_{d,\alpha}\}$ there holds
	\[\left\Vert \sum_{\tau\sim\tau'\subset\mathcal{A}_0^j}\mathcal{E}_{\alpha}f_{\tau}\mathcal{E}_{\alpha}f_{\tau'}\right\Vert_{L^p(\mathbb{R}^{d+1})}^{p_*}\lesssim\sum_{\tau\sim\tau'\subset\mathcal{A}_0^j}\left\Vert \mathcal{E}_{\alpha}f_{\tau}\mathcal{E}_{\alpha}f_{\tau'}\right\Vert_{L^p(\mathbb{R}^{d+1})}^{p_*}, \quad p_*:=\min(p,p')\]
	for all functions $f\in L^2(\bR^d)$ with their supports $\mathrm{supp}f\subset \mathcal{A}_0^j$.
\end{lem}
\begin{lem}[Tao's bilinear restriction estimate, see also Theorem 2.7 of \cite{DY2024}]\label{11} 
	Tao's bilinear restriction estimate to the truncated surface implies the following fact: for $p>\frac{d+3}{d+1}$ and $\tau\sim\tau'\subset\mathcal{A}_0$, there holds
	\[\Vert \mathcal{E}_{\alpha}f_{\tau}\mathcal{E}_{\alpha}f_{\tau'}\Vert_{L^p(\mathbb{R}^{d+1})}\lesssim|\tau|^{1-\frac{d+2}{dp}}\Vert f_{\tau}\Vert_{L^2(\mathbb{R}^d)}\Vert f_{\tau'}\Vert_{L^2(\mathbb{R}^d)}, \quad \forall f\in L^2(\mathbb{R}^d).\]
\end{lem}
\begin{lem}[Proposition 2.2 of \cite{DY2024}]\label{1}
	Let $d\geq 2$ and $\alpha> 2$. There exists $\theta\in(0,1)$ such that the following refined Strichartz estimate holds
	\[\left\Vert \widetilde{\mathcal{E}}_{\alpha} f\right\Vert_{L^{2+\frac{4}{d}}(\mathbb{R}^{d+1})}\lesssim \left(\sup\limits_{\tau\in\mathcal{D}}|\tau|^{-\frac{1}{2}}\left\Vert\mathcal{E}_{\alpha}f_{\tau}\right\Vert_{L^{\infty}(\mathbb{R}^{d+1})}\right)^{\theta}\Vert f\Vert_{L^2(\mathbb{R}^d)}^{1-\theta}, \quad \forall f\in L^2(\mathbb{R}^d).\]
\end{lem}
\begin{lem}[Lemma 3.2 of \cite{DY2024}]\label{2}
	For a bounded sequence of functions ${f_n}$ in $L^2(\mathbb{R}^d)$, which satisfies $f_n\rightharpoonup0$ weakly in $L^2(\mathbb{R}^d)$ as $n\to\infty$, then up to a subsequence we have the following pointwise estimates
	\[\lim\limits_{n\to\infty}\mathcal{E}_{\alpha}f_n(x,t)=0,\quad \lim\limits_{n\to\infty}\widetilde{\mathcal{E}}_{\alpha}f_n(x,t)=0, \quad\quad \text{a.e.} \;\; (x,t)\in \mathbb{R}^{d+1}.\]
\end{lem}

To illustrate the effect of frequency translation parameters and their asymptotic behavior, we recall some notations in \cite[Section 4]{DY2024}. For $f\in L^2(\mathbb{R}^d)$ and $\xi_0\in\mathbb{R}^d$, we introduce the operators
\[\displaystyle T_{\alpha}(\xi_0)f(x,t):=\int_{\mathbb{R}^d}|\xi+\xi_0|^{\frac{\alpha-2}{2+\frac{4}{d}}} e ^{ i (x,t)\cdot(\xi+\xi_0,|\xi+\xi_0|^{\alpha})}f(\xi)\mathrm{d}\xi\]
and
\[\bar{T}_{\alpha}(\xi_0)f(x,t):=\int_{\mathbb{R}^d}\left|\dfrac{\xi}{|\xi_0|}+\dfrac{\xi_0}{|\xi_0|}\right|^{\frac{\alpha-2}{2+\frac{4}{d}}} e ^{ i (x,t)\left(\xi,\frac{|\xi+\xi_0|^{\alpha}-|\xi_0|^{\alpha}-\alpha|\xi_0|^{\alpha-2}\xi_0\cdot\xi}{|\xi_0|^{\alpha-2}}\right)}f(\xi)\mathrm{d}\xi.\]
Then one can check
\begin{equation}\label{neq18}
	\left\Vert \widetilde{\mathcal{E}}_{\alpha}\left(f(\cdot-\xi_0)\right) \right\Vert_{L^{2+\frac{4}{d}}(\mathbb{R}^{d+1})}=\left\Vert T_{\alpha}(\xi_0)f\right\Vert_{L^{2+\frac{4}{d}}(\mathbb{R}^{d+1})}=\left\Vert \bar{T}_{\alpha}(\xi_0)f\right\Vert_{L^{2+\frac{4}{d}}(\mathbb{R}^{d+1})}.
\end{equation}
For an unit vector $\xi_0\in\mathbb{S}^{d-1}$, we define the linear transformation $A_0$ on 
$\mathbb{R}^d$ as follows \[A_0: \xi\mapsto \sqrt{\dfrac{\alpha}{2}} \xi^{\bot} + \sqrt{\dfrac{\alpha(\alpha-1)}{2}} \xi^{\parallel}, \quad \xi^{\parallel} := (\xi\cdot\xi_0)\xi_0, \quad \xi^{\bot}:=\xi-\xi^{\parallel};\]
meanwhile, define the associated unitary operator $\widetilde{A_0}$ on $L^2(\mathbb{R}^d)$ by $\widetilde{A_0}f(x):=|\det A_0|^{1/2}f(A_0x)$.
\begin{lem}[Lemma 4.1 of \cite{DY2024}]\label{4}
	Let $f\in L^2(\mathbb{R}^d)$ with $\Vert f\Vert_{L^2(\mathbb{R}^d)}=1$ and $\{\xi_n\}\subset\mathbb{R}^d$ with $|\xi_n|\to\infty$ as $n\to\infty$. Set the notations
	\[\xi_0 :=\lim\limits_{n\to\infty} {\xi_n}/{|\xi_n|}, \quad f_n :=f(\xi-\xi_n), \quad {\mathbf{\mathrm{a}}}^*_{d,\alpha}:= (\alpha-1)^{-\frac{1}{2d+4}} \left({\alpha}/{2}\right)^{-\frac{d}{2d+4}}.\]
	Then, up to subsequences, we have the following asymptotic Schrödinger behavior
	$$\lim\limits_{n\to\infty}\left\Vert \widetilde{\mathcal{E}}_{\alpha} f_n\right\Vert_{L^{2+\frac{4}{d}}(\mathbb{R}^{d+1})}=\lim\limits_{n\to\infty}\left\Vert T_{\alpha}(\xi_n)f\right\Vert_{L^{2+\frac{4}{d}}(\mathbb{R}^{d+1})}={\mathbf{\mathrm{a}}}^*_{d,\alpha}\left\Vert\mathcal{E}_2\widetilde{A_0}^{-1} f\right\Vert_{L^{2+\frac{4}{d}}(\mathbb{R}^{d+1})}.$$
\end{lem}
We point out that this Lemma \ref{4} relies on the dominated convergence theorem and the classical stationary phase analysis, which gives the following decay estimate
\begin{equation}\label{eq1}
	\left|\bar{T}_{\alpha}(\xi_n)f(x,t)\right|\lesssim\left|\int_{\mathbb{R}^d} e ^{ i t\frac{|\xi+\xi_n|^{\alpha}-|\xi_n|^{\alpha}-\alpha|\xi_n|^{\alpha-2}\xi_n\cdot\xi}{|\xi_n|^{\alpha-2}}+ i x\cdot\xi}f(\xi)\mathrm{d}\xi\right|\lesssim \left(1+t^2+|x|^2\right)^{-\frac{d}{4}}
\end{equation}
for sufficiently large $n$. As we will see later, this estimate \eqref{eq1} implies that a sequence of functions which looks like $\{f_n\}$ in Lemma \ref{4} cannot be an extremal sequence for $M_{\alpha,p}$; on the other hand, based on the work of \cite[Theorem 1.4]{Shao2009EJDE} which claims the existence of extremals for $M_{2,2}$, this Lemma \ref{4} immediately implies the fact $\widetilde{M}_{\alpha}\geq  \mathbf{\mathrm{a}}^*_{d,\alpha} M_{2,2}$.

\begin{lem}[Lemma 4.4 of \cite{DY2024}]\label{5}
	If ${\xi_n}\subset\mathbb{R}^d$ with $|\xi_n|\to\infty$ as $n\to\infty$, and a bounded sequence of functions ${f_n}$ in $L^2(\mathbb{R}^d)$ satisfies $\mathrm{supp}f_n\subset\left\{\xi:|\xi|\leq {|\xi_n|}/{5}\right\}$, and $f_n\rightharpoonup0$ weakly in $L^2(\mathbb{R}^d)$ as $n\to\infty$, then up to subsequences we have $T_{\alpha}(\xi_n)f_n\to0$ strongly in $L^2_{\mathrm{loc}}(\mathbb{R}^{d+1})$, and hence
	$$\lim_{n\to\infty}T_{\alpha}(\xi_n)f_n(x,t)=0,\quad \text{a.e.} \;\; (x,t)\in \mathbb{R}^{d+1}.$$
\end{lem}

Finally, by the endpoint Strichartz estimate of Keel-Tao \cite[Theorem 1.2]{KT1998} and some stationary phase arguments, one can obtain the following Sobolev-type Strichartz estimate. Further details can be seen, for example, in the paper \cite[Proof of Theorem 1.2]{COX2011}.
\begin{lem} \label{9}
	For all parameters $(d,q,r,s)$ satisfying
	\[\frac{d}{r}+\frac{2}{q}\leq \frac{d}{2},\quad q,r\geq 2,\quad s=\frac{d}{2}-\frac{\alpha}{q}-\frac{d}{r},\]
	there holds
	$$\Vert\mathcal{E}_{\alpha}f\Vert_{L^q(\mathbb{R})L^r(\mathbb{R}^d)}\lesssim\left\Vert |\cdot|^sf\right\Vert_{L^2(\mathbb{R}^d)}.$$
\end{lem}

\subsection{The $L^p$ consequences} \label{SubS:Lp preliminaries}
As mentioned in the introduction, the $L^p$ extremals for Fourier extension type estimates have been studied for several situations such as the paraboloid \cite{Stovall2020}, sphere \cite{FS2024}, cone \cite{NOST2023}, and moment curve \cite{BS2023}. Here we summarize some of their results, which play fundamental roles in our arguments.

\begin{lem}[Lemma 2.3 of \cite{Stovall2020}] \label{L:Xsb to chip refinement}
	Assume that $s<p<2t$ and $\max\{{p}/{2t}, {s}/{p}\}<\theta<1$. Then for the function $\Vert f\Vert_{L^p(\mathbb{R}^d)}=1$, there exists $c_0>0$ such that
	\[\left(\sum_k\sum_{\tau\in\mathcal{D}_k}|\tau|^{-2t(\frac{1}{s}-\frac{1}{p})}\Vert f_{\tau}\Vert_{L^s(\mathbb{R}^d)}^{2t}\right)^{\frac{1}{2t}} \lesssim\sup_{k\in\mathbb{Z}}\sup_{\tau\in\mathcal{D}_k}\sup_{l\geq 0}2^{-c_0l}\Vert f_{\tau,l}\Vert_{L^p(\mathbb{R}^d)}^{\theta},\]
	where $f_{\tau,l}$ equals $f$ multiplied by the characteristic function of $\tau \cap \{\xi: |f|< 2^l \|f\|_p |\tau|^{-1/p}\}$.
\end{lem}

\begin{lem}[Lemma 3.3 of \cite{NOST2023}] \label{L:Bilinear to linear refinement}
	For two measure spaces $(X,\mu)$ and $(Y,\nu)$, consider the bounded linear map $T:L^p(X)\to L^q(Y)$ with $1<p<q<\infty$. Let $\{P_j\}_{j\in\mathbb{Z}}$ be a sequence of bounded linear operators on $L^p(X)$ such that $\sum_{j\in\mathbb{Z}}P_j$ converges to the identity in strong topology and
	\[\sum_{j\in\mathbb{Z}}\Vert P_jf\Vert_{L^p(X)}^p\lesssim \Vert f\Vert_{L^p(X)}^p, \quad \forall f\in L^p(X).\]
	Assume that there exists $c_0>0$ such that
	\[\l\| (TP_jf)(TP_kg) \r\|_{L^{\frac{q}{2}}(Y)} \lesssim 2^{-c_0|j-k|}\Vert f\Vert_{L^p(X)}\Vert g\Vert_{L^p(X)}, \quad \forall (j,k) \in \bZ^2\]
	holds for arbitrary $L^p(X)$ functions $f$ and $g$. Then there exists $\theta \in(0,1)$ such that
	\[\Vert Tf\Vert_{L^q(Y)} \lesssim \sup_{j\in\mathbb{Z}}\Vert TP_jf\Vert_{L^q(Y)}^{1-\theta}\Vert f\Vert_{L^p(X)}^{1-\theta}, \quad \forall f\in L^p(X).\]
\end{lem}

\begin{lem}[Lemma 4.3 of \cite{Stovall2020}]\label{17}
	Assume that $\varphi$ and $\psi$ are compactly supported smooth nonnegative functions satisfying  $\|\varphi\|_{L^{\infty}(\bR^d)}=\Vert \psi\Vert_{L^1(\mathbb{R}^d)}=1$. Consider the parameters
	\[\{(x_n^j,t_n^j):n\in\mathbb{N},j\in\mathbb{N}\}\subset\mathbb{R}^{d+1}, \quad \lim\limits_{n\to\infty}(|x_n^j-x_n^k|+|t_n^j-t_n^k|)=\infty,\quad \forall j\ne k.\]
	For the integers $j\in\mathbb{N}$ and $J\in \bN$, define the following two projection-type operators
	\[\pi(x_n^j,t_n^j)f(\xi) :=e^{-i(x_n^j,t_n^j)(\xi,|\xi|^{\alpha})} \left(\psi* \l(e^{i(x_n^j,t_n^j)(\cdot,|\cdot|^{\alpha})} \varphi f\r)\right)(\xi), \quad \Pi_n^J:=\l(\pi(x_n^j,t_n^j)\r)_{j=1}^J\]
	Then for each $J\in\bN$, we have the following estimate
	\[\limsup\limits_{n\to\infty}\Vert\Pi_n^J\Vert_{L^p\to \ell^{p^*}(L^p)}\leq 1, \quad p^*:= \max \{p,p'\}.\]
\end{lem}

\section{The $L^2$ theory} \label{ns1}
The $L^2$-based profile decomposition is one of the main ingredients to show the desired $L^p$-precompact result Theorem \ref{n1}. Thus, in this section, we investigate the $L^2$ case first. As mentioned above, Di-Yan \cite{DY2024} has investigated the $L^2$-precompactness by the missing mass method and established some preliminaries. Here we will show the desired $L^2$-based profile decomposition in Subsection \ref{SubS:L2-profile decomposition}, and then give some corollaries in Subsection \ref{SubS:Coro for L2-profile decomposition}.

\subsection{The $L^2$-based profile decomposition} \label{SubS:L2-profile decomposition}
To establish the desired profile decomposition, one main tool is the following inverse Strichartz estimate.
\begin{lem}[Inverse Strichartz estimate]\label{7}
	Let $\alpha> 2$ and $\{f_n\}$ be a bounded sequence in $L^2(\mathbb{R}^d)$ with $d\geq 2$. Assuming 
	$$\limsup_{n\to\infty}\left\Vert \widetilde{\mathcal{E}}_{\alpha} f_n\right\Vert_{L^{2+\frac{4}{d}}(\mathbb{R}^{d+1})}\ne0,$$
	then up to subsequences there exist
	$$\{(\xi_n,h_n,x_n,t_n)\}\subset\mathbb{R}^d\times\mathbb{R}_+\times\mathbb{R}^d\times\mathbb{R},\quad\quad \phi\in L^2(\mathbb{R}^d),$$
	which satisfy the following weak convergence
	\[\mathcal{G}^{-1}(h_n,x_n,t_n)f_n\left(\cdot+\frac{\xi_n}{h_n}\right)\rightharpoonup\phi \quad \text{in} \;\; L^2(\mathbb{R}^d) \quad \text{as}\;\; n\to\infty\]
	and the following orthogonality
	\begin{equation}\label{neq8}
		\Vert f_n\Vert_{L^2(\mathbb{R}^d)}^2=\Vert \phi\Vert_{L^2(\mathbb{R}^d)}^2+\Vert \omega_n\Vert_{L^2(\mathbb{R}^d)}^2+o_{n\to\infty}(1), \quad \omega_n:=f_n-\mathcal{G}(h_n,x_n,t_n)\phi\left(\cdot-\frac{\xi_n}{h_n}\right).
	\end{equation}
	Moreover, if we further quantitatively assume
	$$\Vert f_n\Vert_{L^2(\mathbb{R}^d)}\leq B, \quad \limsup\limits_{n\to\infty}\left\Vert \widetilde{\mathcal{E}}_{\alpha} f_n\right\Vert_{L^{2+\frac{4}{d}}(\mathbb{R}^{d+1})}=A,$$
	then we have the lower bound estimate $A^{\frac{1}{\theta}}B^{1-\frac{1}{\theta}}\lesssim\Vert\phi\Vert_{L^2(\mathbb{R}^d)}$ and the following orthogonality
	$$\left\Vert\widetilde{\mathcal{E}}_{\alpha} f_n\right\Vert_{L^{2+\frac{4}{d}}(\mathbb{R}^{d+1})}^{2+\frac{4}{d}}=\left\Vert\widetilde{\mathcal{E}}_{\alpha}\left(\phi\left(\cdot-\frac{\xi_n}{h_n}\right)\right)\right\Vert_{L^{2+\frac{4}{d}}(\mathbb{R}^{d+1})}^{2+\frac{4}{d}}+\left\Vert\widetilde{\mathcal{E}}_{\alpha}\omega_n\right\Vert_{L^{2+\frac{4}{d}}(\mathbb{R}^{d+1})}^{2+\frac{4}{d}}+o_{n\to\infty}(1),$$
	as well as the following Strichartz norm estimate
	\begin{equation}\label{neq7}
		\limsup_{n\to\infty}\left\Vert\widetilde{\mathcal{E}}_{\alpha}\omega_n\right\Vert_{L^{2+\frac{4}{d}}(\mathbb{R}^{d+1})}^{2+\frac{4}{d}}\leq A^{2+\frac{4}{d}}\left(1-C\left(\dfrac{A}{B}\right)^{(2+\frac{4}{d})(\frac{1}{\theta}-1)}\right),
	\end{equation}
	where the constant $C$ depends only on $(d,\alpha, \theta)$ and $\theta$ comes from the Lemma \ref{1}. 
\end{lem}
\begin{proof}
	According to Lemma \ref{1}, there exist $\{\tau_n\}\subset \mathcal{D}$ and $\{(x_n,t_n)\}\subset \mathbb{R}^{d+1}$ such that 
	$$\left\Vert \widetilde{\mathcal{E}}_{\alpha} f_n\right\Vert_{L^{2+\frac{4}{d}}(\mathbb{R}^{d+1})}\lesssim\left(|\tau_n|^{-\frac{1}{2}}\left|\left(\mathcal{E}_{\alpha}{f_n}_{\tau_n}\right)(x_n,t_n)\right|\right)^{\theta}\Vert f\Vert_{L^2(\mathbb{R}^d)}^{1-\theta}.$$
	Denoting $\tau_n=\xi_n+{h_n}\left[1,2\right]^d$ with $h_n>0$, then one have
	\begin{equation}\label{neq3}
		\left(\dfrac{\left\Vert \widetilde{\mathcal{E}}_{\alpha} f_n\right\Vert_{L^{2+\frac{4}{d}}(\mathbb{R}^{d+1})}}{\Vert f_n\Vert_{L^2(\mathbb{R}^d)}^{1-\theta}}\right)^{\frac{1}{\theta}}
		\lesssim|\tau_n|^{-\frac{1}{2}}\left|\left(\mathcal{E}_{\alpha}{f_n}_{\tau_n}\right)(x_n,t_n)\right|=h_n^{-\frac{d}{2}}\left|\int_{\tau_n} e ^{ i (x_n,t_n)(\xi,|\xi|^{\alpha})}f_n(\xi)\mathrm{d}\xi\right|.
	\end{equation}
	Denoting $\xi=\xi_n+h_n\xi'$, then direct computation can give
	\begin{equation}\label{neq4}
		h_n^{-\frac{d}{2}}\left|\int_{\tau_n} e ^{ i (x_n,t_n)(\xi,|\xi|^{\alpha})}f_n(\xi)\mathrm{d}\xi\right|=\left|\int_{[1,2]^d}\mathcal{G}^{-1}(h_n,x_n,t_n)f_n\left(\xi'+\frac{\xi_n}{h_n}\right)\mathrm{d}\xi'\right|.
	\end{equation}
	Notice the following sequence 
	$$\mathcal{G}^{-1}(h_n,x_n,t_n)f_n\left(\cdot+\frac{\xi_n}{h_n}\right)$$
	is bounded in $L^2(\mathbb{R}^d)$. By the Alaoglu theorem, up to a subsequence, there exists $\phi\in L^2(\bR^d)$ such that
	\begin{equation}\label{neq6}
		\mathcal{G}^{-1}(h_n,x_n,t_n)f_n\left(\cdot+\frac{\xi_n}{h_n}\right)\rightharpoonup\phi,\quad \text{as} \;\; n\to\infty \quad \text{in} \;\; L^2(\mathbb{R}^d).
	\end{equation}
	then the orthogonality \eqref{neq8} holds.
	\par 
	If
	$$\Vert f_n\Vert_{L^2(\mathbb{R}^d)}\leq B
	\quad\mbox{and}\quad\limsup\limits_{n\to\infty}\left\Vert \widetilde{\mathcal{E}}_{\alpha} f_n\right\Vert_{L^{2+\frac{4}{d}}(\mathbb{R}^{d+1})}=A.$$
	Letting $n\to\infty$, Hölder's inequality, inequality \eqref{neq3} and equation \eqref{neq4} imply that
	\begin{equation}\label{neq5}
		A^{\frac{1}{\theta}}B^{1-\frac{1}{\theta}}\lesssim \left|\int_{\mathbb{R}^d}\phi(\xi)\chi_{[1,2]^d}(\xi)\mathrm{d}\xi\right|\leq \Vert \phi\Vert_{L^2(\mathbb{R}^d)}.
	\end{equation}
	Now, we divide the proof into two cases.
	
	\emph{Case 1:} up to subsequences $\left\{\xi_n/{h_n}\right\}$ is bounded. There exist $\xi_0\in \mathbb{R}^d$ such that ${\xi_n}/{h_n}\to\xi_0$ as $n\to\infty$. We can then let $\xi_n=0$ and regard $\phi\left(\cdot-\xi_0\right)$ as the new $\phi$, which gives that
	$$\mathcal{G}^{-1}(h_n,x_n,t_n)f_n\rightharpoonup\phi \quad \text{in}\;\; L^2(\bR^d) \quad\text{as}\;\; n\to\infty.$$
	According to Lemma \ref{2}, up to subsequences $$\lim_{n\to\infty}\widetilde{\mathcal{E}}_{\alpha}\mathcal{G}^{-1}(h_n,x_n,t_n)f_n(x,t)= \widetilde{\mathcal{E}}_{\alpha}\phi(x,t),\quad \text{a.e.} \;\; (x,t)\in \mathbb{R}^{d+1}.$$
	By the classical Br\'ezis-Lieb lemma \cite[Section 1.9]{LL2001}, we have
	\begin{equation}\label{eq3}
		\left\Vert \widetilde{\mathcal{E}}_{\alpha} f_n\right\Vert_{L^{2+\frac{4}{d}}(\mathbb{R}^{d+1})}^{2+\frac{4}{d}}=\left\Vert \widetilde{\mathcal{E}}_{\alpha}\phi\right\Vert_{L^{2+\frac{4}{d}}(\mathbb{R}^{d+1})}^{2+\frac{4}{d}}+\left\Vert \widetilde{\mathcal{E}}_{\alpha} \omega_n\right\Vert_{L^{2+\frac{4}{d}}(\mathbb{R}^{d+1})}^{2+\frac{4}{d}}+o_{n\to\infty}(1).
	\end{equation}
	Let
	\[h(\xi,\lambda)= |\xi|^{-\frac{\alpha-2}{2+\frac{4}{d}}} \chi_{[1,2]^d}(\xi) \chi_{[0,(2\sqrt{d})^{\alpha}]} (\lambda)\]
	then $h\in L^1(\mathbb{R}^{d+1})$, and by direct calculation $\hat{h}\in L^{\left(2+\frac{4}{d}\right)'}(\mathbb{R}^{d+1})$, where 
	$$\hat{h}(x,t):=(2\pi)^{-\frac{d+1}{2}}\int_{\mathbb{R}^{d+1}} e ^{- i (x,t)(\xi,\lambda)}h(\xi,\lambda)\mathrm{d}\xi \mathrm{d}\lambda,$$
	is the space-time Fourier transform of $h$. There is 
	\begin{equation*}
		\begin{aligned}
			\left|\int_{\mathbb{R}^{d+1}}\widetilde{\mathcal{E}}_{\alpha}\phi(x,t)\hat{h}(x,t)\mathrm{d}x \mathrm{d}t\right|
			=&\left|\int_{\mathbb{R}^{d+1}}\left(\widetilde{\mathcal{E}}_{\alpha}\phi\right)^{\wedge}(\xi,\lambda)h(\xi,\lambda)\mathrm{d}\xi \mathrm{d}\lambda\right|\\
			=&(2\pi)^{\frac{d+1}{2}}\left|\int_{\mathbb{R}^{d+1}}\delta\left(\lambda-|\xi|^{\alpha}\right)|\xi|^{\frac{\alpha-2}{2+\frac{4}{d}}}\phi(\xi)h(\xi,\lambda)\mathrm{d}\xi \mathrm{d}\lambda\right|\\
			=&(2\pi)^{\frac{d+1}{2}}\left|\int_{\mathbb{R}^d}|\xi|^{\frac{\alpha-2}{2+\frac{4}{d}}}\phi(\xi)h(\xi,|\xi|^{\alpha})\mathrm{d}\xi \right|\\
			=&(2\pi)^{\frac{d+1}{2}}\left|\int_{\mathbb{R}^d}\phi(\xi)\chi_{[1,2]^d}(\xi)\mathrm{d}\xi\right|
		\end{aligned}
	\end{equation*}
	Then by inequality \eqref{neq5}, 
	$$A^{\frac{1}{\theta}}B^{1-\frac{1}{\theta}}\lesssim\left\Vert \widetilde{\mathcal{E}}_{\alpha}\phi\right\Vert_{L^{2+\frac{4}{d}}(\mathbb{R}^{d+1})}.$$
	Combining this with equation \eqref{eq3}, we obtain $$\limsup_{n\to\infty}\left\Vert \widetilde{\mathcal{E}}_{\alpha} \omega_n\right\Vert_{L^{2+\frac{4}{d}}(\mathbb{R}^{d+1})}^{2+\frac{4}{d}}\leq A^{2+\frac{4}{d}}\left(1-C\left(\dfrac{A}{B}\right)^{(2+\frac{4}{d})(\frac{1}{\theta}-1)}\right).$$
	
	\emph{Case 2}: there holds $\lim\limits_{n\to\infty}\frac{|\xi_n|}{h_n}= \infty$. Denoting
	$$E_n:=\left\{\xi\in\mathbb{R}^d:|\xi|\leq \dfrac{|\xi_n|}{5h_n}\right\}$$
	and
	$$r_n:=\mathcal{G}^{-1}(h_n,x_n,t_n)\omega_n\left(\cdot+\frac{\xi_n}{h_n}\right)\chi_{E_n},\quad q_n:=\mathcal{G}^{-1}(h_n,x_n,t_n)\omega_n\left(\cdot+\frac{\xi_n}{h_n}\right)\chi_{E_n^c},$$
	then there holds
	$$\left\Vert\widetilde{\mathcal{E}}_{\alpha} f_n\right\Vert_{L^{2+\frac{4}{d}}(\mathbb{R}^{d+1})}=\left\Vert \bar{T}_{\alpha}\left(\frac{\xi_n}{h_n}\right)(\phi+r_n+q_n)\right\Vert_{L^{2+\frac{4}{d}}(\mathbb{R}^{d+1})}.$$
	Notice that equation \eqref{neq6} implies $r_n\rightharpoonup0$ in $L^2(\mathbb{R}^d)$ as $n\to\infty$, according to Lemma \ref{5}, up to subsequences
	$$\lim_{n\to\infty}\bar{T}_{\alpha}\left(\frac{\xi_n}{h_n}\right)r_n(x,t)=0$$
	almost everywhere in $\mathbb{R}^{d+1}$. Combining $q_n\to0$ in $L^2(\mathbb{R}^d)$ as $n\to\infty$ with equation \eqref{neq18} and estimate \eqref{neq17}, we see that $$\bar{T}_{\alpha}\left(\frac{\xi_n}{h_n}\right)q_n\to0 \quad \text{in} \;\; L^{2+\frac{4}{d}}(\mathbb{R}^{d+1}) \quad \text{as} \;\; n\to\infty.$$
	Recall the inequality \eqref{eq1} which gives the desired dominating function in $L^{2+\frac{4}{d}}(\bR^{d+1})$. Combining these three facts, a generalization of the Br\'ezis-Lieb lemma \cite[Lemma 3.1]{FLS2016} implies
	\begin{equation*}
		\begin{aligned}
			&\left\Vert \bar{T}_{\alpha}\left(\frac{\xi_n}{h_n}\right)(\phi+r_n+q_n)\right\Vert_{L^{2+\frac{4}{d}}(\mathbb{R}^{d+1})}^{2+\frac{4}{d}}\\
			=&\left\Vert \bar{T}_{\alpha}\left(\frac{\xi_n}{h_n}\right)\phi\right\Vert_{L^{2+\frac{4}{d}}(\mathbb{R}^{d+1})}^{2+\frac{4}{d}}+\left\Vert \bar{T}_{\alpha}\left(\frac{\xi_n}{h_n}\right)r_n\right\Vert_{L^{2+\frac{4}{d}}(\mathbb{R}^{d+1})}^{2+\frac{4}{d}}+o_{n\to\infty}(1)\\
			=&\left\Vert\widetilde{\mathcal{E}}_{\alpha}\left(\phi\left(\cdot-\frac{\xi_n}{h_n}\right)\right)\right\Vert_{L^{2+\frac{4}{d}}(\mathbb{R}^{d+1})}^{2+\frac{4}{d}}+\left\Vert\widetilde{\mathcal{E}}_{\alpha}\omega_n\right\Vert_{L^{2+\frac{4}{d}}(\mathbb{R}^{d+1})}^{2+\frac{4}{d}}+o_{n\to\infty}(1),
		\end{aligned}
	\end{equation*} 
	Finally, similar to the previous \emph{Case 1}, we conclude
	\[\limsup_{n\to\infty}\left\Vert\widetilde{\mathcal{E}}_{\alpha}\omega_n\right\Vert_{L^{2+\frac{4}{d}}(\mathbb{R}^{d+1})}^{2+\frac{4}{d}}\leq A^{2+\frac{4}{d}}\left(1-C\left(A/B\right)^{(2+\frac{4}{d})(\frac{1}{\theta}-1)}\right).\]
	This completes the proof.
\end{proof}

Applying this lemma repeatedly, one can obtain the following profile decomposition.
\begin{definition}[Orthogonal parameters]
	For a sequence of parameters $\{(\xi_n^j,h_n^j,x_n^j,t_n^j)\}$, we say the parameters are \textit{orthogonal} if they satisfy: for each fixed $j \neq k$, either
	\[\lim\limits_{n\to\infty}\left(\frac{h_n^k}{h_n^j}+\frac{h_n^j}{h_n^k}+\frac{\left|\xi_n^j-\xi_n^k\right|}{h_n^k}\right)=\infty\]
	or
	\[(\xi_n^j,h_n^j)\equiv (\xi_n^k,h_n^k) , \quad \lim\limits_{n\to\infty} \left(\frac{|\xi_n^j|}{h_n^j} + h_n^j \left|x_n^j-x_n^k\right| + \left(h_n^j\right)^{\alpha} \left|t_n^j-t_n^k\right| \right)=\infty.\]
\end{definition}
\begin{pro}[Profile decomposition]\label{8}
	Let $\alpha> 2$ and $\{f_n\}$ be a bounded sequence in $L^2(\mathbb{R}^d)$ with $d\geq 2$. Then, up to a subsequence, there exist orthogonal parameters $\{(\xi_n^j,h_n^j,x_n^j,t_n^j)\} \subset \mathbb{R}^d\times\mathbb{R}_+ \times\mathbb{R}^d \times\mathbb{R}$, $J_0\in\mathbb{N}_+ \cup \{\infty\}$, and  $\{\phi^j\} \subset L^2(\mathbb{R}^d)$ with $\phi^j\ne0$, 
	so that for all integers $J\leq J_0$, we have the following profile decomposition
	$$f_n=\sum_{j=1}^J\mathcal{G}(h_n^j,x_n^j,t_n^j)\left(\phi^j\left(\cdot-\frac{\xi_n^j}{h_n^j}\right)\right)+\omega_n^J$$
	with the following properties: the weak convergence for profiles
	\begin{equation}\label{neq15}
		\mathcal{G}^{-1}(h_n^j,x_n^j,t_n^j)f_n\left(\cdot+\frac{\xi_n^j}{h_n^j}\right)\rightharpoonup\phi^j \quad \text{in} \;\; L^2(\mathbb{R}^d) \quad \text{as} \;\; n\to\infty;
	\end{equation}
	the orthogonality of $L^2$ norms
	\begin{equation}\label{eq6}
		\Vert f_n\Vert_{L^2(\mathbb{R}^d)}^2=\sum_{j=1}^J\Vert \phi^j\Vert_{L^2(\mathbb{R}^d)}^2+\Vert \omega_n^J\Vert_{L^2(\mathbb{R}^d)}^2+o_{n\to\infty}(1);
	\end{equation}
	the orthogonality of $L^{2+4/d}$ norms
	\begin{equation}\label{eq7}
		\left\Vert \widetilde{\mathcal{E}}_{\alpha} f_n\right\Vert_{L^{2+\frac{4}{d}}(\mathbb{R}^{d+1})}^{2+\frac{4}{d}}
		=\sum_{j=1}^J\left\Vert \widetilde{\mathcal{E}}_{\alpha}\left(\phi^j\left(\cdot-\dfrac{\xi_n^j}{h_n^j}\right)\right) \right\Vert_{L^{2+\frac{4}{d}}(\mathbb{R}^{d+1})}^{2+\frac{4}{d}}+\left\Vert \widetilde{\mathcal{E}}_{\alpha} \omega_n^J\right\Vert_{L^{2+\frac{4}{d}}(\mathbb{R}^{d+1})}^{2+\frac{4}{d}}+o_{n\to\infty}(1);
	\end{equation}
	and the smallness of remainder $L^{2+4/d}$ norms
	\begin{equation}\label{eq8}
		\lim_{J\to J_0}\limsup_{n\to\infty}\left\Vert \widetilde{\mathcal{E}}_{\alpha}\omega_n^J\right\Vert_{L^{2+\frac{4}{d}}(\mathbb{R}^{d+1})}=0.
	\end{equation}
\end{pro}
\begin{proof}
	Assume that 
	$$\Vert f_n\Vert_{L^2(\mathbb{R}^d)}\leq B_0
	\quad\mbox{and}\quad\limsup\limits_{n\to\infty}\left\Vert \widetilde{\mathcal{E}}_{\alpha} f_n\right\Vert_{L^{2+\frac{4}{d}}(\mathbb{R}^{d+1})}=A_0.$$
	If $A_0=0$, let $\phi^j=0,\omega_n^1 =f_n$, otherwise, using Proposition \ref{7}, up to subsequences there exist
	$$\{(\xi_n^1,h_n^1,x_n^1,t_n^1)\}\subset\mathbb{R}^d\times\mathbb{R}_+\times\mathbb{R}^d\times\mathbb{R},\quad \phi^1\in L^2(\mathbb{R}^d),$$
	such that 
	$$\mathcal{G}^{-1}(h_n^1,x_n^1,t_n^1)f_n\left(\cdot+ \frac{\xi_n^1}{h_n^1}\right)\rightharpoonup\phi^1,\quad n\to\infty$$
	in weak topology of $L^2(\mathbb{R}^d)$,  let $\omega_n^1=f_n-\mathcal{G}(h_n^1,x_n^1,t_n^1)\left(\phi^1\left(\cdot-\frac{\xi_n^1}{h_n^1}\right)\right)$, then 
	$$\Vert f_n\Vert_{L^2(\mathbb{R}^d)}^2=\Vert \phi^1\Vert_{L^2(\mathbb{R}^d)}^2+\Vert \omega_n^1\Vert_{L^2(\mathbb{R}^d)}^2+o_{n\to\infty}(1).$$
	$$A_0^{\frac{1}{\theta}}B_0^{1-\frac{1}{\theta}}\lesssim\Vert\phi^1\Vert_{L^2(\mathbb{R}^d)},$$
	$$\left\Vert\widetilde{\mathcal{E}}_{\alpha} f_n\right\Vert_{L^{2+\frac{4}{d}}(\mathbb{R}^{d+1})}^{2+\frac{4}{d}}=\left\Vert\widetilde{\mathcal{E}}_{\alpha}\left(\phi^1\left(\cdot-\frac{\xi_n^1}{h_n^1}\right)\right)\right\Vert_{L^{2+\frac{4}{d}}(\mathbb{R}^{d+1})}^{2+\frac{4}{d}}+\left\Vert\widetilde{\mathcal{E}}_{\alpha}\omega_n^1\right\Vert_{L^{2+\frac{4}{d}}(\mathbb{R}^{d+1})}^{2+\frac{4}{d}}+o_{n\to\infty}(1),$$
	and 
	$$\limsup_{n\to\infty}\left\Vert\widetilde{\mathcal{E}}_{\alpha}\omega_n^1\right\Vert_{L^{2+\frac{4}{d}}(\mathbb{R}^{d+1})}^{2+\frac{4}{d}}\leq A_0^{2+\frac{4}{d}}\left(1-C\left(\dfrac{A_0}{B_0}\right)^{(2+\frac{4}{d})(\frac{1}{\theta}-1)}\right).$$
	\par Now, we replace $f_n$ with $\omega_n^1$, and repeat the same process until $\limsup\limits_{n\to\infty}\left\Vert\widetilde{\mathcal{E}}_{\alpha}\omega_n^j\right\Vert_{L^{2+\frac{4}{d}}(\mathbb{R}^{d+1})}=0$, let this $j$ be $J_0$. Otherwise, let $J_0=\infty$. We obtain orthogonality \eqref{eq6} and \eqref{eq7} immediately.
	\par If $J_0<\infty$, equation\eqref{eq8} is trivial. Otherwise let $B_J=\sup\limits_{n\in\mathbb{N}}\Vert \omega_n^J\Vert_{L^2(\mathbb{R}^d)}$, $A_J=\limsup\limits_{n\to\infty}\left\Vert \widetilde{\mathcal{E}}_{\alpha} \omega_n^J\right\Vert_{L^{2+\frac{4}{d}}(\mathbb{R}^{d+1})}$, inequality \eqref{neq7} implies that 
	$$A_{J+1}^{2+\frac{4}{d}}\leq A_J^{2+\frac{4}{d}}\left(1-C\left(\dfrac{A_J}{B_J}\right)^{(2+\frac{4}{d})(\frac{1}{\theta}-1)}\right).$$
	By equation \eqref{eq6}, we have $B_J\leq B_0$. Thus
	$$A_{J+1}^{2+\frac{4}{d}}\leq A_J^{2+\frac{4}{d}}\left(1-C\left(\dfrac{A_J}{B_J}\right)^{(2+\frac{4}{d})(\frac{1}{\theta}-1)}\right)\leq A_J^{2+\frac{4}{d}}\left(1-C\left(\dfrac{A_J}{B_0}\right)^{(2+\frac{4}{d})(\frac{1}{\theta}-1)}\right).$$
	Noticing that $\{A_J\}$ is decreasing, let $J\to\infty$, we have $\lim\limits_{J\to\infty}A_J=0$.
	
	Define the operator $\widetilde{g^j_n}f:=\mathcal{G}^{-1}(h_n^j,x_n^j,t_n^j)f\left(\cdot+\frac{\xi_n^j}{h_n^j}\right)$.	Then we have
	\begin{equation}\label{neq9}
		\widetilde{g^1_n}\omega_n^1\rightharpoonup0,\quad \widetilde{g^2_n}\omega_n^1\rightharpoonup\phi^2\ne0 \quad \text{in} \;\; L^2(\mathbb{R}^d) \quad \text{as} \;\; n\to\infty.
	\end{equation}
	By direct calculation, one can obtain that $\widetilde{g^1_n}\omega_n^1=\left(\widetilde{g^1_n}\widetilde{g^2_n}^{-1}\right)\widetilde{g^2_n}\omega_n^1$ is equal to
	\begin{equation}\label{neq10}
		\left(\frac{h_n^1}{h_n^2}\right)^{\frac{d}{2}} e ^{ i \left(h_n^1\left(x_n^1-x_n^2\right),\left(h_n^1\right)^{\alpha}\left(t_n^1-t_n^2\right)\right)\left(\cdot+\frac{\xi_n^1}{h_n^1},\left|\cdot +\frac{\xi_n^1}{h_n^1}\right|^{\alpha}\right)}\widetilde{g^2_n}\omega_n^1\left(\frac{h_n^1}{h_n^2}\cdot+\frac{\xi_n^1-\xi_n^2}{h_n^2}\right).
	\end{equation}
	Therefore, if the following relation occurs
	\begin{equation}\label{neq11}
		\lim\limits_{n\to\infty}\left(\frac{h_n^2}{h_n^1}+\frac{h_n^1}{h_n^2}+\frac{1}{h_n^1}\left|\xi_n^1\right|+\frac{\left|\xi_n^1-\xi_n^2\right|}{h_n^2}+h_n^1\left|x_n^1-x_n^2\right|+\left(h_n^1\right)^{\alpha}\left|t_n^1-t_n^2\right|\right)<\infty,
	\end{equation}
	there is a contradiction between \eqref{neq9} and \eqref{neq10}. Thus the sequence must tend to infinity. At that time, if 
	$$\lim\limits_{n\to\infty}\left(\frac{h_n^2}{h_n^1}+\frac{h_n^1}{h_n^2}+\frac{\left|\xi_n^1-\xi_n^2\right|}{h_n^2}\right)=\infty.$$
	Let $\phi,\psi\in C_c^{\infty}(\mathbb{R}^d)$, assume that $\mathrm{supp}\phi,\mathrm{supp}\psi\subset B(0,R)$. Let 
	$$B_n^1=B\left(\frac{\xi_n^2-\xi_n^1}{h_n^1},\frac{h_n^2}{h_n^1}R\right),\quad B_n^2=B\left(\frac{\xi_n^1-\xi_n^2}{h_n^2},\frac{h_n^1}{h_n^2}R\right).$$
	By Hölder's inequality,
	$$\left|\left\langle \left(\widetilde{g^2_n}\widetilde{g^1_n}^{-1}\right)\phi,\psi\right\rangle\right|=\left|\left\langle \left(\widetilde{g^2_n}\widetilde{g^1_n}^{-1}\right)\phi,\psi\chi_{B_n^2}\right\rangle\right|\leq\Vert\phi\Vert_{L^2(\mathbb{R}^{d})}\Vert\psi\chi_{B_n^2}\Vert_{L^2(\mathbb{R}^{d})},$$
	and 
	$$\left|\left\langle \left(\widetilde{g^2_n}\widetilde{g^1_n}^{-1}\right)\phi,\psi\right\rangle\right|=\left|\left\langle \phi\chi_{B_n^1},\left(\widetilde{g^1_n}\widetilde{g^2_n}^{-1}\right)\psi\right\rangle\right|\leq\Vert\phi\chi_{B_n^1}\Vert_{L^2(\mathbb{R}^{d})}\Vert\psi\Vert_{L^2(\mathbb{R}^{d})},$$
	by the dominated convergence theorem, $\lim\limits_{n\to\infty}\left|\left\langle \left(\widetilde{g^2_n}\widetilde{g^1_n}^{-1}\right)\phi,\psi\right\rangle\right|=0$.  Combining this with $C_c^{\infty}(\mathbb{R}^d)$ is dense in $L^2(\mathbb{R}^d)$, we obtain  $\left(\widetilde{g^2_n}\widetilde{g^1_n}^{-1}\right)\phi^1\rightharpoonup 0$ and $\widetilde{g^2_n}f_n\rightharpoonup \phi^2$ weakly in $L^2(\mathbb{R}^d)$ as $n\to\infty$.
	Otherwise, 
	$$\lim\limits_{n\to\infty}\left(\frac{1}{h_n^1}\left|\xi_n^1\right|+h_n^1\left|x_n^1-x_n^2\right|+\left(h_n^1\right)^{\alpha}\left|t_n^1-t_n^2\right|\right)=\infty,$$
	and we may assume that $(\xi_n^1,h_n^1)=(\xi_n^2,h_n^2)$. If $\lim\limits_{n\to\infty}\left(\frac{1}{h_n^1}\left|\xi_n^1\right|\right)=\infty$, similarly  there is also $\left(\widetilde{g^2_n}\widetilde{g^1_n}^{-1}\right)\phi^1\rightharpoonup 0$ and $\widetilde{g^2_n}f_n\rightharpoonup \phi^2$ weakly in $L^2(\mathbb{R}^d)$ as $n\to\infty$. If $\lim\limits_{n\to\infty}\left(h_n^1\left|x_n^1-x_n^2\right|+\left(h_n^1\right)^{\alpha}\left|t_n^1-t_n^2\right|\right)=\infty$, by stationary phase method, we also obtain the weak convergence. Finally, by iteration, one can obtain the desired conclusions \eqref{neq15} and the orthogonality of parameters. This completes the proof.
\end{proof}

\subsection{Some applications of the profile decomposition} \label{SubS:Coro for L2-profile decomposition}
This subsection contains some corollaries of the profile decomposition Proposition \ref{8} obtained in the previous subsection.

First, combining the aforementioned profile decomposition and the asymptotic Schr\"odinger behavior Lemma \ref{4}, one can follow the classical arguments in \cite[Proof of Theorem 1.4]{Shao2009EJDE} to deduce the following characterization of precompactness for extremal sequences.
\begin{cor}[Characterization of precompactness]
	Let $\alpha>2$ and $d\geq 2$. Then all extremal sequences for $\widetilde{M}_{\alpha}$ are precompact up to symmetries if and only if $$\widetilde{M}_{\alpha}> \mathbf{\mathrm{a}}^*_{d,\alpha} M_{2,2}.$$
\end{cor}
Second, applying the Sobolev embedding and stationary phase analysis, one can follow the arguments in \cite[Proof of Theorem 2.4]{HS2012} to establish the following noncritical profile decomposition.
\begin{cor}[Noncritical profile decomposition]\label{16}
	Let $\alpha> 2$ and $\{f_n\}$ be a bounded sequence in $L^2(\mathbb{R}^d)$ with $d\geq 2$. Then up to subsequences, there exist orthogonal parameters $\{(h_n^j,x_n^j,t_n^j)\}\subset\mathbb{R}_+\times\mathbb{R}^d\times\mathbb{R}$, $J_0\in\mathbb{N}_+ \cup\{\infty\}$, and $\{\phi^j\}\subset L^2(\mathbb{R}^d)$ with $\phi^j\ne0$, 
	such that for all $J\leq J_0$, we have the following profile decomposition
	\begin{equation*}
		f_n=\sum_{j=1}^J \mathcal{G}(h_n^j,x_n^j,t_n^j)\phi^j+\omega_n^J,
	\end{equation*}
	with the following properties: the weak convergence for profiles
	\begin{equation}\label{neq16}
		\mathcal{G}^{-1}(h_n^j,x_n^j,t_n^j)f_n\rightharpoonup\phi^j,\quad \text{as} \;\; n\to\infty \quad \text{in} \;\;L^2(\mathbb{R}^d);
	\end{equation}
	for each $J\in\bN$, the orthogonality of $L^2$ norms
	\begin{equation}\label{eq12}
		\Vert f_n\Vert_{L^2(\mathbb{R}^d)}^2=\sum_{j=1}^J\Vert \phi^j\Vert_{L^2(\mathbb{R}^d)}^2+\Vert \omega_n^J\Vert_{L^2(\mathbb{R}^d)}^2+o_{n\to\infty}(1);
	\end{equation}
	the orthogonality of $L^{2+2\alpha/d}$ norms
	\begin{equation}\label{eq13}
		\left\Vert \mathcal{E}_{\alpha} f_n\right\Vert_{L^{2+\frac{2\alpha}{d}}(\mathbb{R}^{d+1})}^{2+\frac{2\alpha}{d}}=\sum_{j=1}^J\left\Vert \mathcal{E}_{\alpha}\phi^j \right\Vert_{L^{2+\frac{2\alpha}{d}}(\mathbb{R}^{d+1})}^{2+\frac{2\alpha}{d}}+\left\Vert \mathcal{E}_{\alpha} \omega_n^J\right\Vert_{L^{2+\frac{2\alpha}{d}}(\mathbb{R}^{d+1})}^{2+\frac{2\alpha}{d}}+o_{n\to\infty}(1);
	\end{equation}
	and the smallness of $L^{2+2\alpha/d}$ norms
	\begin{equation}\label{eq14}
		\lim_{J\to J_0}\limsup_{n\to\infty}\left\Vert \mathcal{E}_{\alpha} \omega_n^J\right\Vert_{L^{2+\frac{2\alpha}{d}}(\mathbb{R}^{d+1})}=0.
	\end{equation}
\end{cor}
Third, as this profile decomposition Corollary \ref{16} does not contain the frequency translation parameters $\xi_n$, one can follow the arguments in \cite[Proof of Theorem 1.2]{HS2012} to derive the following precompactness result.
\begin{cor}[Precompactness of noncritical extremal sequences] \label{C:p=2 noncritical}
	Let $\alpha> 2$ and $d\geq 2$. Then all extremal sequences for $M_{\alpha,2}$ are precompact up to symmetries, thus extremals exist.
\end{cor}

\section{The $L^p$ convergence}\label{ns2}
In this section, we show the desired $L^p$-precompactness result Theorem \ref{n1}. For the case $p=1$, one can see that $M_{\alpha,1}=1$, hence the existence of extremals and the non-compactness of extremal sequences are elementary. Recall that the typical $p=2$ case has been solved in Corollary \ref{C:p=2 noncritical}. Thus throughout this section, we investigate the non-endpoint case $p\in(1,2)$ with $q=\frac{d+\alpha}{d}p'$.

In Subsection \ref{SubS:Frequency localization}, we establish some $L^p$-refined Strichartz estimates and show some frequency localization properties for the extremal sequences; then in Subsection \ref{SubS:Space-time localization}, we apply the $L^2$-theory established in the previous section to show some space-time localization properties for the extremal sequences; finally in Subsection \ref{SubS:The Lp convergence}, we show the desired $L^p$-precompactness Theorem \ref{n1}.

\subsection{Frequency localization}\label{SubS:Frequency localization}
We first demonstrate that extremal sequences possess good frequency localization, which relies on some refinements of the original Strichartz estimate \eqref{neq1}.
\begin{lem}[Annular refinement]\label{14}
	There exists a parameter $\theta_0 \in (0,1)$ such that 
	\[\Vert \mathcal{E}_{\alpha}f\Vert_{L^q(\mathbb{R}^{d+1})}\lesssim\sup_{N\in\mathbb{Z}}\Vert \mathcal{E}_{\alpha}f_{\mathcal{A}_N}\Vert_{L^q(\mathbb{R}^{d+1})}^{\theta_0}\Vert f\Vert_{L^p(\mathbb{R}^d)}^{1-\theta_0}, \quad \forall f\in L^p(\mathbb{R}^d).\]
\end{lem}
\begin{proof}
	We adapt some ideas from \cite[Proof of Lemma 3.2]{NOST2023}, which investigates the cone case. By Lemma \ref{L:Bilinear to linear refinement}, it suffices to show that there exists $c_1>0$ and the following bilinear estimate
	\begin{equation*}
		\Vert \mathcal{E}_{\alpha} f_{\mathcal{A}_{N_1}} \mathcal{E}_{\alpha} f_{\mathcal{A}_{N_2}} \Vert_{L^{\frac{q}{2}}(\mathbb{R}^{d+1})} \lesssim 2^{-c_1|N_1-N_2|} \Vert f_{\mathcal{A}_{N_1}} \Vert_{L^p(\mathbb{R}^d)} \Vert f_{\mathcal{A}_{N_2}} \Vert_{L^p(\mathbb{R}^d)}.
	\end{equation*}
	Without loss of generality, we may assume $N_1>N_2$. Then we choose two pairs $(r_1,q_1,s_1)$ and $(r_2,q_2,s_2)$ which satisfy the following relations
	\[\dfrac{d}{r_1}+\dfrac{2}{q_1}\leq \dfrac{d}{2},\quad \dfrac{d}{r_2}+\dfrac{2}{q_2}\leq \dfrac{d}{2},\quad \frac{d}{d+\alpha}= \frac{1}{q_1} +\frac{1}{q_2} = \frac{1}{r_1} +\frac{1}{r_2},\]
	and
	\[s_1=\dfrac{d}{2}-\dfrac{\alpha}{q_1}-\dfrac{d}{r_1},\quad s_2=\dfrac{d}{2}-\dfrac{\alpha}{q_2}-\dfrac{d}{r_2}>0.\]
	Here one can check the existence of these two pairs since $\alpha>2$, and check the identity $s_1=-s_2$. Then applying the Hölder inequality and Lemma \ref{9}, we conclude
	\begin{equation*}
		\begin{aligned}
			\Vert \mathcal{E}_{\alpha}f_{\mathcal{A}_{N_1}} \mathcal{E}_{\alpha}f_{\mathcal{A}_{N_2}} \Vert_{L^{\frac{d+\alpha}{d}}(\mathbb{R}^{d+1})}
			\leq &\Vert \mathcal{E}_{\alpha}f_{\mathcal{A}_{N_1}}\Vert_{L^{q_1}(\mathbb{R})L^{r_1}(\mathbb{R}^d)}\Vert \mathcal{E}_{\alpha}f_{\mathcal{A}_{N_2}}\Vert_{L^{q_2}(\mathbb{R})L^{r_2}(\mathbb{R}^d)}\\
			\lesssim&\Vert |\cdot|^{-s_2}f_{\mathcal{A}_{N_1}}\Vert_{L^2(\mathbb{R}^d)}\Vert|\cdot|^{s_2} f_{\mathcal{A}_{N_2}}\Vert_{L^2(\mathbb{R}^d)}\\
			\lesssim &2^{-s_2(N_1-N_2)}\Vert f_{\mathcal{A}_{N_1}}\Vert_{L^2(\mathbb{R}^d)}\Vert f_{\mathcal{A}_{N_2}}\Vert_{L^2(\mathbb{R}^d)}.
		\end{aligned}
	\end{equation*}
	Recalling that $\mathcal{E}_{\alpha}$ is $L^1(\bR^d)$ to $L^{\infty}(\bR^{d+1})$ bounded, we obtain
	\[\Vert \mathcal{E}_{\alpha} f_{\mathcal{A}_{N_1}} \mathcal{E}_{\alpha} f_{\mathcal{A}_{N_2}} \Vert_{L^{\infty}(\mathbb{R}^{d+1})} \lesssim \Vert f_{\mathcal{A}_{N_1}} \Vert_{L^1(\mathbb{R}^d)} \Vert f_{\mathcal{A}_{N_2}} \Vert_{L^1(\mathbb{R}^d)}.\]
	Finally, by interpolating the aforementioned two estimates on $L^2$ and $L^1$, we obtain the desired bilinear estimate and complete the proof.
\end{proof}

\begin{lem}[Chip refinement]\label{21}
	There exist $\theta\in (0,1)$ and $c_0>0$ such that 
	\begin{equation}\label{eq18}
		\Vert \mathcal{E}_{\alpha}f\Vert_{L^q(\mathbb{R}^{d+1})}\lesssim \sup\limits_{k\in\mathbb{Z}}\sup\limits_{\tau\in \mathcal{D}_{k}}\sup\limits_{l\geq 0}2^{-c_0l}\Vert f_{\tau,l}\Vert_{L^p(\bR^d)}^{\theta}\Vert f\Vert_{L^p(\bR^d)}^{1-\theta}, \quad \forall f\in L^p(\mathbb{R}^d),
	\end{equation}
	where $f_{\tau,l}$ equals $f$ multiplied by the characteristic function of $\tau\cap\{\xi:|f(\xi)|<2^l\Vert f\Vert_p|\tau|^{-1/p}\}$.
\end{lem}
\begin{proof}
	For arbitrary $q_1>\frac{2(d+3)}{d+1}$, Tao's bilinear restriction estimate (see Lemma \ref{11}) implies that 
	\begin{equation}\label{eq2}
		\Vert \mathcal{E}_{\alpha}f_{\tau}\mathcal{E}_{\alpha}f_{\tau'}\Vert_{L^{\frac{q_1}{2}}(\mathbb{R}^{d+1})}\lesssim|\tau|^{1-2\frac{d+2}{dq_1}}\Vert f_{\tau}\Vert_{L^2(\mathbb{R}^d)}\Vert f_{\tau'}\Vert_{L^2(\mathbb{R}^d)}\lesssim |\tau|^{1-2\frac{d+\alpha}{dq_1}}\Vert f_{\tau}\Vert_{L^2(\mathbb{R}^d)}\Vert f_{\tau'}\Vert_{L^2(\mathbb{R}^d)}
	\end{equation}
	holds for all $f\in L^2(\mathbb{R}^d)$ with $\mathrm{supp}f\subset \mathcal{A}_0$ and for all $\tau\sim\tau'\subset\mathcal{A}_0$. Here the second inequality is followed by $|\tau|\leq 1$ and $\alpha> 2$. For arbitrary $p_2\in [1,2]$ and $q_2=\frac{d+\alpha}{d}p_2'$, $\mathcal{E}_{\alpha}$ is bounded from $L^p(\mathbb{R}^d)$ to $L^{q}(\mathbb{R}^{d+1})$. Thus by Hölder's inequality, we can conclude that
	\begin{equation}\label{eq4}
		\Vert \mathcal{E}_{\alpha}f_{\tau}\mathcal{E}_{\alpha}f_{\tau'}\Vert_{L^{\frac{q_2}{2}}(\mathbb{R}^{d+1})}\lesssim\Vert f_{\tau}\Vert_{L^{p_2}(\mathbb{R}^d)}\Vert f_{\tau'}\Vert_{L^{p_2}(\mathbb{R}^d)}.
	\end{equation}
	We can choose suitable $q_1$ and $q_2$ such that $p$ lies between  $q_1$ and $q_2$. Interpolating these two bilinear estimates \eqref{eq2} and \eqref{eq4}, there exists $s<p$ such that 
	\begin{equation}\label{eq16}
		\Vert \mathcal{E}_{\alpha}f_{\tau}\mathcal{E}_{\alpha}f_{\tau'}\Vert_{L^{\frac{q}{2}}(\mathbb{R}^{d+1})}\lesssim|\tau|^{-2(\frac{1}{s}-\frac{1}{p})}\Vert f_{\tau}\Vert_{L^{s}(\mathbb{R}^d)}\Vert f_{\tau'}\Vert_{L^{s}(\mathbb{R}^d)}
	\end{equation}
	holds for all $f\in L^p(\mathbb{R}^d)$ with $\mathrm{supp}f\subset\mathcal{A}_0$. Then for any $j\in\{1,2,\cdots,K_{d,\alpha}\}$, applying the Whitney decomposition and Lemma \ref{12}, we have
	$$\l\Vert \mathcal{E}_{\alpha}f_{\mathcal{A}_0^j} \r\Vert_{L^q(\mathbb{R}^{d+1})}^2 \leq \l\Vert \sum_{\tau\sim\tau'\subset\mathcal{A}_0^j} \mathcal{E}_{\alpha}f_{\tau} \mathcal{E}_{\alpha}f_{\tau'} \r\Vert_{L^{\frac{q}{2}}(\mathbb{R}^{d+1})} \lesssim \l(\sum_{\tau\sim\tau'\subset\mathcal{A}_0^{j_0}} \left\Vert \mathcal{E}f_{\tau} \mathcal{E}f_{\tau'} \right\Vert_{L^{\frac{q}{2}}(\mathbb{R}^{d+1})}^{\theta_1} \r)^\frac{1}{\theta_1},$$
	where $\theta_1:=(q/2)_*>p/2>s/2$ and recall $(q/2)_*:= \min\{q/2,(q/2)'\}$. By \eqref{eq16}, we further have
	$$\Vert \mathcal{E}_{\alpha}f_{\tau}\mathcal{E}_{\alpha}f_{\tau'}\Vert_{L^{\frac{q}{2}}(\mathbb{R}^{d+1})}\lesssim|\tau|^{-2(\frac{1}{s}-\frac{1}{p})}\Vert f_{\tau}\Vert_{L^{s}(\mathbb{R}^d)}\Vert f_{\tau'}\Vert_{L^{s}(\mathbb{R}^d)}\lesssim |\tau''|^{-2(\frac{1}{s}-\frac{1}{p})}\Vert f_{\tau''}\Vert_{L^s(\mathbb{R}^d)}^2,$$
	where $\tau''$ is a slightly larger cube containing both $\tau$ and $\tau'$. Thus we conclude
	$$\left\Vert \mathcal{E}_{\alpha}f_{\mathcal{A}_0^j}\right\Vert_{L^q(\mathbb{R}^{d+1})}^{2}\lesssim \l( \sum_k\sum_{\tau\in\mathcal{D}_{k}} |\tau|^{-2\theta_1 (\frac{1}{s}-\frac{1}{p})} \Vert f_{\tau}\Vert_{L^s(\mathbb{R}^d)}^{2\theta_1} \r)^{\frac{1}{\theta_1}}.$$
	Applying Lemma \ref{L:Xsb to chip refinement}, we can deduce
	\begin{equation*}
		\begin{aligned}
			\left\Vert \mathcal{E}_{\alpha}f_{\mathcal{A}_0}\right\Vert_{L^q(\mathbb{R}^{d+1})}
			\lesssim &\max_{1\leq j\leq K_{d,\alpha}} \left\Vert\mathcal{E}_{\alpha}f_{\mathcal{A}_0^j}\right\Vert_{L^p(\mathbb{R}^{d+1})}\\
			\lesssim &\l(\sum_k\sum_{\tau\in\mathcal{D}_{k}}|\tau|^{-2 \theta_1 (\frac{1}{s}-\frac{1}{p})}\Vert f_{\tau}\Vert_{L^s(\mathbb{R}^d)}^{2\theta_1} \r)^{\frac{1}{2\theta_1}}\\
			\lesssim &\sup_{k\in\mathbb{Z}}\sup_{\tau\in\mathcal{D}_{k}}\sup_{l\geq 0}2^{-c_1l}\Vert f_{\tau,l}\Vert_{L^p(\mathbb{R}^d)}^{\theta_2}\Vert f\Vert_{L^p(\mathbb{R}^d)}^{1-\theta_2}
		\end{aligned}
	\end{equation*}
	for some $c_1>0$ and $0<\theta_2<1$. Then by rescaling, we obtain
	$$\left\Vert \mathcal{E}_{\alpha}f_{\mathcal{A}_N}\right\Vert_{L^q(\mathbb{R}^{d+1})}
	\lesssim\sup_{k\in\mathbb{Z}}\sup_{\tau\in\mathcal{D}_{k}}\sup_{l\geq 0}2^{-c_1l}\Vert f_{\tau,l}\Vert_{L^p(\mathbb{R}^d)}^{\theta_2}\Vert f\Vert_{L^p(\mathbb{R}^d)}^{1-\theta_2}$$
	holds for all $f\in L^p(\mathbb{R}^d)$ with $\mathrm{supp}f\subset\mathcal{A}_N$. Combining this estimate and Lemma \ref{14}, we get the desired estimate \eqref{eq18} and finish the proof.
\end{proof}

\begin{lem}[Chip extraction]\label{15}
	There exists a decreasing sequence $\rho_j\to0$ as $j\to\infty$, such that for every $f\in L^p(\mathbb{R}^d)$, there exists a sequence $\tau^j$ of dyadic cubes satisfying that: if the sequences $(g^j,r^j)$ are defined inductively by the following
	\begin{equation}\label{L:Chip extraction-1}
		r^0:=f,\quad f^j:=r^{j-1}\chi_{\tau^j}\chi_{\{\xi:|f(\xi)|<2^j |\tau^j|^{-{1}/{p}} \Vert f\Vert_{L^p(\mathbb{R}^d)}\}},\quad r^j:=r^{j-1}-f^j,
	\end{equation}
	then there holds
	\[\Vert\mathcal{E}_{\alpha}h^j\Vert_{L^q(\mathbb{R}^{d+1})}\leq \rho_j\Vert f\Vert_{L^p(\mathbb{R}^d)}, \quad\quad  \forall |h^j|=|r^j|\chi_{E} \;\; \text{with} \;\;\forall E\subset \bR^d.\]
\end{lem}
\begin{proof}
	We are going to show that the desired decreasing sequence can be taken as $\rho_j:= C_{*} (j/2)^{-\theta/p}$, where $\theta$ is taken from Lemma \ref{21} and $C_{*}$ is the optimal constant of the inequality \eqref{eq18}. Then for fixed $j$, we investigate the following non-zero item
	\[\sup\limits_{\tau\in\mathcal{D}}\Vert f\chi_{\tau}\chi_{\{\xi:|f(\xi)|<2^j |\tau|^{-{1}/{p}} \|f\|_{L^p(\mathbb{R}^d)}\}}\Vert_{L^p(\mathbb{R}^d)}. \]
	Let the cube sequence $\tau_n= c_n + h_n[0,1]^d$ approach this supremum (bigger than half of this supremum). Then one can use the dominated convergence theorem to show that the cases $h_n\to \infty$ and $h_n\to 0$ cannot occur, and moreover, the case $c_n\to\infty$ cannot occur. These imply the fact that the aforementioned supremum can be reached.
		
	Now the proof is shown by induction. Without loss of generality, assume that $\Vert f\Vert_{L^p(\mathbb{R}^d)}=1$. Due to the fact discussed above, given $r^{j-1}$,  we can take $\tau^j$ such that
	\begin{equation} \label{E:Chip extraction-1}
		\|r^{j-1}\chi_{\tau^j}\chi_{\{\xi:|f(\xi)|<2^j |\tau^j|^{-{1}/{p}}\}} \|_{L^p(\mathbb{R}^d)} = \sup_{\tau} \|r^{j-1} \chi_{\tau}\chi_{\{\xi:|f(\xi)|<2^j |\tau|^{-{1}/{p}}\}} \|_{L^p(\mathbb{R}^d)},
	\end{equation}
	and then define the corresponding functions $(f^j, r^j)$ as shown in \eqref{L:Chip extraction-1}. This procedure leads to the sequence $(f^j,r^j)$. Then we introduce the following notation
	\[A_j:= \sup_{\tau\in \mathcal{D}} \|r^{2j} \chi_{\tau} \chi_{\{\xi:|f(\xi)|<2^{j} |\tau|^{-{1}/{p}}\}} \|_{L^p(\mathbb{R}^d)}.\]
	Since we are taking $\tau^j$ to be maximal in each step, for $j\leq j' \leq 2j$, there holds $A_j \leq \|f^{j'}\|_{L^p(\mathbb{R}^d)}$. Further noticing that by our construction the supports of $f^j$ are disjoint, we conclude
	\[j A_j^p \leq \sum_{k=1}^j \|f^{j+k}\|_{L^p(\mathbb{R}^d)}^p \leq 1, \quad \Rightarrow \quad A_j\leq j^{-1/p}.\]
	Then for any function $|h^{2j}|=|r^{2j}|\chi_E$ with some set $E$, by the Lemma \ref{21}, we obtain
	\[\|\mathcal{E}_{\alpha} h^{2j}\|_{L^q(\mathbb{R}^{d+1})} \leq C_{*} \sup_{j\geq 0} 2^{-c_0j} A_j^{\theta} \leq C_{*} \sup_{j\geq 0} 2^{-c_0j} j^{-\theta/p} \leq C_{*} j^{-\theta/p}.\]
	Therefore, we can choose $\rho_j:= C_{*} (j/2)^{-\theta/p}$ as mentioned above and completes the proof.
\end{proof}

One can observe that the functions $f^j$ in Lemma \ref{15}, even with small support, may have large values, and conversely, $f^j$ with small values may have wide support. Next, we aim to control the support and the value simultaneously.
\begin{pro}[Frequency localization]\label{18}
	Let $\{f_n\}$ be an extremal sequence for $M_{\alpha,p}$. Then there exist dilation parameters $\left\{h_n\right\} \subset \mathbb{R}_+$ such that
	\[\lim_{m\to\infty}\limsup_{n\to\infty}\left\Vert \left(h_n\right)^{\frac{d}{p}}f_n\left(h_n\cdot\right)\chi_{\left(\{\xi:|\xi|>m\}\cup\left\{\xi:\left|\left(h_n\right)^{{d}/{p}}f_n\left(h_n\cdot\right)\right|>m\right\}\right)}\right\Vert_{L^p(\mathbb{R}^{d})}=0.\]
\end{pro}
\begin{proof}
	If this proposition were to fail, for any dilation parameters $\left\{h_n\right\}\subset \mathbb{R}_+$, there would exist $\varepsilon>0$ and $\lim_{l\to\infty} m_l=\infty$ such that 
	\begin{equation}\label{E:Frequency localization-1}
		\limsup_{n\to\infty}\left\Vert \left(h_n\right)^{\frac{d}{p}}f_n\left(h_n\cdot\right)\chi_{\left(\{\xi:|\xi|>m_l\}\cup\left\{\xi:\left|\left(h_n\right)^{\frac{d}{p}}f_n\left(h_n\cdot\right)\right|>m_l\right\}\right)}\right\Vert_{L^p(\mathbb{R}^d)}>\varepsilon
	\end{equation}
	for every $l\in\mathbb{N}$. Then by Lemma \ref{15}, we can take $J\in \bN$ large enough (to be determined later) and dyadic cubes $\tau_n^j$ such that
	\[f_n= \sum_{j=1}^J f_n^j + r_n^J, \quad \quad \quad \|\mathcal{E}_{\alpha} r_n^J\|_{L^q(\mathbb{R}^{d+1})} < \rho_J , \quad \quad \lim_{J \to \infty} \rho_J =0.\]
	Denote $F_n^J:= f_n - r_n^J$. Then the disjointness for the supports of $F_n^J$ and $r_n^J$ implies that
	\[M_{\alpha,p}-\rho_J \leq \liminf_{n\to\infty} \|\mathcal{E}_{\alpha} F_n^J\|_{L^q(\mathbb{R}^{d+1})} \leq M_{\alpha,p} (1-\|r_n^J\|_{L^p(\mathbb{R}^d)}^p)^{1/p} \leq M_{\alpha,p} - c_p \limsup_{n\to\infty} \|r_n^J\|_{L^p(\mathbb{R}^d)}^p.\]
	Hence, up to subsequences, we can assume
	\begin{equation}\label{neq19}
		\|r_n^J\|_{L^p(\mathbb{R}^d) }\lesssim \rho_J^{1/p}.
	\end{equation}
	Noticing that there exists $J_*$ such that $\rho_{J_*} <M_{\alpha,p}$, thus up to subsequences, there exists one index $j_0 \in\{1,2,\ldots ,J_*\}$ such that
	\begin{equation}\label{neq20}
		\Vert\mathcal{E}_{\alpha}f_n^{j_0}\Vert_{L^q(\mathbb{R}^{d+1})}\geq \frac{M_{\alpha,p}-\rho_{J_*}}{2J_*}>0, \quad\forall n\in \mathbb{N}.
	\end{equation}
	Without loss of generality, we may assume $j_0 =1$ and then introduce the following notations
	\[\tau_n^1 := \xi_n^1 + 2^{-k_n^1} [0,1]^d, \quad g_n^1(\xi) :=2^{\frac{k_n^1 d}{p}}f_n^1\left(2^{-k_n^1}\xi\right), \quad \Rightarrow \quad \l| g_n^1\left(\xi-2^{k_n^1}\xi_n^1\right) \r| \leqslant 2\chi_{[0,1]^d}(\xi).\]
	If there holds $2^{k_n^1}|\xi_n^1|\to\infty$ as $n\to\infty$, recalling the stationary phase consequence \eqref{eq1}, then the fact $\alpha>2$ and dominated convergence theorem imply that
	\begin{equation*}
		\begin{aligned}
			&\l\|\mathcal{E}_{\alpha} f_n^1\r\|_{L^q(\mathbb{R}^{d+1})}^q=\l\|\mathcal{E}_{\alpha} g_n^1\r\|_{L^q(\mathbb{R}^{d+1})}^q \\
			&= \left(2^{k_n^1}|\xi_n^1|\right)^{2-\alpha} \int_{\bR^{d+1}} \l|\int_{\bR^d} e^{ix\xi + it \frac{|\xi-2^{k_n^1}\xi_n^1|^{\alpha} - \left(2^{k_n^1}|\xi_n^1|\right)^{\alpha} + \alpha\left(2^{k_n^1}|\xi_n^1|\right)^{\alpha-2} 2^{k_n^1}\xi\cdot\xi_n^1}{\left(2^{k_n^1}|\xi_n^1|\right)^{\alpha-2}}} g_n^1\left(\xi-2^{k_n^1}\xi_n^1\right)\mathrm{d}\xi\r|^q \ddd x \ddd t \\
			&=o_{n\to\infty}(1),
		\end{aligned}
	\end{equation*}
	which contradicts the inequality \eqref{neq20}. Hence, up to subsequences, there exists $|\xi^1|<\infty$ such that $2^{k_n^1}\xi_n^1\to \xi^1$ as $n\to\infty$. Thus by applying scaling symmetries, we can assume $\tau_n^1 =\xi^1+ [0,1]^d$ for all $n$ and the remainder cubes can be written as
	\[\tau_n^j = \xi_n^j + 2^{-k_n^j} [0,1]^d, \;\;\; k_n^j \in \bZ, \;\;\; \xi_n^j \in 2^{-k_n^j} \bZ^d, \quad\quad n\in \bN, \quad 1\leq j \leq J.\]
	Then by our hypothesis, for such scaling symmetries $\{h_n\}$ we can get $\varepsilon$ and $m_l$ satisfying the inequality \eqref{E:Frequency localization-1}; and by taking $J$ large enough, the inequality \eqref{neq19} can guarantee that
	\begin{equation} \label{E:Frequency localization-2}
		\|f_n-F_n^J\|_{L^p(\mathbb{R}^d)} < \varepsilon/2.
	\end{equation}
	
	For each fixed $j$, up to subsequences for $n$, we may assume either $|k_n^j|\to \infty$ or $|k_n^j|\to k^j$, either $|\xi_n^j|\to \infty$ or $\xi_n^j\to \xi^j$. We say that the index $j$ is \textit{good} if the parameters $(|k_n^j|, \xi_n^j)\to (k^j, \xi^j)$ as $n$ goes to infinity, and say $j$ is \textit{bad} if it is not good. Then we decompose
	\[F_n^J=G_n^J + B_n^J, \quad G_n^J:= \sum_{j \;\text{good}} f_n^j, \quad B_n^J:= F_n^J- G_n^J.\]
	By the choice of $G_n^J$, for $l$ large enough, there holds
	\[A_l \cap \supp G_n^J =\emptyset, \quad\quad A_l:= \l\{\xi: |\xi|>m_l \r\} \cup \l\{\xi: |f_n(\xi)|>m_l \r\}.\]
	Thus, up to subsequences for $n$, we can use \eqref{E:Frequency localization-2} to deduce
	\[\|f_n\|_{L^p(A_l)} = \|B_n^J + r_n^J\|_{L^p(A_l)} >\varepsilon, \quad \Rightarrow\quad \liminf_{n\to\infty} \|B_n^J\|_{L^p(\mathbb{R}^d)} > \varepsilon/2.\]
	Recalling that $F_n^J$ is a portion of $f_n$, we can further deduce
	\[\|G_n^J\|_{L^p(\mathbb{R}^d)} \leq (1-(\varepsilon/2)^p)^{1/p} \leq C_{1,\varepsilon,p} <1.\]
	Then we investigate the item $f_n^1$. Recalling the estimate \eqref{neq20} and the extension estimate \eqref{neq1}, we have $\|f_n^1\|_{L^p}\gtrsim 1$. Also recalling that $\tau_n^1 = \xi^1+[0,1]^d$ for all $n$, this implies that $f_n^1$ is one portion of $G_n^J$. Then by the disjointness of the portions $f_n^k$, we obtain
	\[1\lesssim \|f_n^1\|_{L^p(\mathbb{R}^d)}\leq \|G_n^J\|_{L^p(\mathbb{R}^d)}.\]
	Hence we conclude
	\[\|B_n^J\|_{L^p(\mathbb{R}^d)} = \l(\|F_n^J\|_{L^p(\mathbb{R}^d)}^p - \|G_n^J\|_{L^p(\mathbb{R}^d)}^p\r)^{1/p} \leq C_{2,p}<1.\]
		
	Next we are going to apply the generalized Br\'ezis-Lieb lemma \cite[Lemma 3.1]{FLS2016}, so we need to show that $\mathcal{E}_{\alpha} B_n^J (x,t)\to 0$ almost everywhere in $\bR^{d+1}$ as $n\to \infty$. Since $J$ is a fixed finite number, we only need to investigate each item $\mathcal{E}_{\alpha} f_n^j$ with bad index $j$. For each fixed bad index $j$, we can divide the proof into three cases: if $k_n^j \to +\infty$ as $n$ goes to infinity then $f_n^j \to 0$ in $L^1(\mathbb{R}^d)$, which by direct computation implies $\mathcal{E}_{\alpha} f_n^j \to 0$ almost everywhere; if $k_n^j \to -\infty$  as $n$ goes to infinity then $f_n^j \to 0$ in $L^2(\mathbb{R}^d)$, which implies $\mathcal{E}_{\alpha} f_n^j \to 0$ in $L^{2+\frac{2\alpha}{d}}(\mathbb{R}^{d+1})$ and up to subsequences $\mathcal{E}_{\alpha} f_n^j \to 0$ almost everywhere; if $|k_n^j|$ remains bounded and $|\xi_n^j| \to \infty$ as $n$ goes to infinity, which implies $f_n^j \rightharpoonup 0$ in $L^2(\mathbb{R}^d)$ and further the Lemma \ref{2} implies $\mathcal{E}_{\alpha}f_n^j(t,x) \to 0$ almost everywhere. In summary, for fixed $J$, we obtain the desired pointwise convergence conclusion $\mathcal{E}_{\alpha} B_n^J \to 0$ almost everywhere. By the construction, for fixed $J$, there exists $G^J\in L^p(\mathbb{R}^d)$ such that $G_n^J \to G^J$ in $L^p(\mathbb{R}^d)$ as $n$ goes to infinity. Hence, by the aforementioned generalized Br\'ezis-Lieb lemma and the fact $q>p$, we conclude
	\begin{align*}
		M_{\alpha,p}-\rho_J &\leq \limsup_{n\to\infty} \l\|\mathcal{E}_{\alpha} F_n^J\r\|_{L^q(\mathbb{R}^{d+1})} = \limsup_{n\to\infty} \l(\|\mathcal{E}_{\alpha} G_n^J\|_{L^q(\mathbb{R}^{d+1})}^q + \|\mathcal{E}_{\alpha} B_n^J\|_{L^q(\mathbb{R}^{d+1})}^q\r)^{1/q} \\
		&\leq M_{\alpha,p} \l(\|G_n^J\|_{L^p(\mathbb{R}^d)}^q + \|B_n^J\|_{L^p(\mathbb{R}^d)}^q\r)^{1/q} \leq M_{\alpha,p} \max \{C_{1,\varepsilon,p}, C_{2,p}\}^{1-p/q} \|F_n\|_{L^p(\mathbb{R}^d)}^{1/q} \\
		&\leq M_{\alpha,p} \max \{C_{1,\varepsilon,p}, C_{2,p}\}^{1-p/q}.
	\end{align*}
	Finally, letting $J \to \infty$ will lead to a contradiction, and thus we complete the proof.
\end{proof}

\subsection{Space-time localization}\label{SubS:Space-time localization}
Now, we utilize the aforementioned $L^2$-theory to establish the $L^p$-based profile decomposition for frequency-localized sequences.
\begin{pro}[$L^p$-profile decomposition]\label{19}
	Let $R>0$ and $\{f_n\}$ be a sequence of measurable functions which support on $\{\xi:|\xi|<R\}$ and satisfy $|f_n(\xi)|<R$. Then up to subsequences there exist bounded measurable functions $\phi^j$ supported on $\{\xi:|\xi|<R\}$ and the integer $J_0\in\mathbb{N}_+ \cup \{\infty\}$ as well as the parameters $\{(x_n^j,t_n^j)\}\subset\mathbb{R}^d\times\mathbb{R}_+$, so that for all $J\leq J_0$ we have the following profile decomposition
	$$f_n=\sum_{j=1}^J e^{-i(x_n^j,t_n^j)(\cdot,|\cdot|^\alpha)}\phi^j+\omega_n^J,$$
	with the following properties: the weak convergence for profiles
	\begin{equation}\label{eq30}
		e^{i(x_n^j,t_n^j)(\cdot,|\cdot|^\alpha)}f_n\rightharpoonup\phi^j \quad \text{in} \;\; L^p(\mathbb{R}^d) \quad\text{as}\;\; n\to\infty, \quad \forall j\leq J;
	\end{equation}
	the orthogonality of parameters
	\begin{equation}\label{eq26}
		\lim_{n\to\infty}(|x_n^j-x_n^k|+|t_n^j-t_n^k|)=\infty,\quad \forall j\ne k;
	\end{equation}
	for each fixed $J$, the quasi-orthogonality of $L^p$ norms
	\begin{equation}\label{eq27}
		\liminf_{n\to\infty}\left(\Vert f_n\Vert_{L^p(\mathbb{R}^d)}^{p'}-\sum_{j=1}^J\Vert \phi^j\Vert_{L^p(\mathbb{R}^d)}^{p'}\right)\geq 0;
	\end{equation}
	the orthogonality of $L^q$ norms
	\begin{equation}\label{eq28}
		\left\Vert \mathcal{E}_{\alpha} f_n\right\Vert_{L^q(\mathbb{R}^{d+1})}^q=\sum_{j=1}^J\left\Vert \mathcal{E}_{\alpha}\phi^j \right\Vert_{L^q(\mathbb{R}^{d+1})}^q+\left\Vert \mathcal{E}_{\alpha} \omega_n^J\right\Vert_{L^q(\mathbb{R}^{d+1})}^q+o_{n\to\infty}(1);
	\end{equation}
	and the smallness of the remainder $L^q$ norms
	\begin{equation}\label{eq29}
		\lim_{J\to J_0}\limsup_{n\to\infty}\left\Vert \mathcal{E}_{\alpha} \omega_n^J\right\Vert_{L^q(\mathbb{R}^{d+1})}=0.
	\end{equation}
\end{pro}
\begin{proof}
	Applying Corollary \ref{16}, up to subsequences there exist orthogonal parameters and
	\[\{(\xi_n^j,h_n^j,x_n^j,t_n^j)\} \subset \mathbb{R}^d\times\mathbb{R}_+ \times\mathbb{R}^d \times\mathbb{R}, \quad J_0\in\mathbb{N}_+ \cup \{\infty\}, \quad \{\phi^j\} \subset L^2(\mathbb{R}^d), \quad \phi^j\ne0,\]
	such that for all $J\leq J_0$ we have the following decomposition
	\begin{equation*}
		f_n=\sum_{j=1}^J \mathcal{G}(h_n^j,x_n^j,t_n^j)\phi^j+\omega_n^J,
	\end{equation*}
	which satisfies the properties \eqref{neq16}-\eqref{eq14}. By the size and support assumptions on the $f_n$, considering the definition of $\phi^j \neq 0$, we ensure that the dilation parameters remain bounded away from $0$ and $\infty$. Hence for each $j$, upon selecting a subsequence, the dilation symmetries converge in the strong operator topology. Hence by redefining $\phi^j$ and then putting the error terms into the remainders, we may assume that
	\begin{equation} \label{E:Profile decomposition-Lp-1}
		f_n=\sum_{j=1}^J e^{-i(x_n^j,t_n^j)(\cdot,|\cdot|^\alpha)}\phi^j+\omega_n^J.
	\end{equation}
	Meanwhile, the desired conclusions \eqref{eq26} and \eqref{eq30} ensue. 
		
	Next we prove the $L^q$-orthogonality \eqref{eq28}. By Lemma \ref{2}, the weak convergence \eqref{neq16} implies the following pointwise estimate
	\begin{equation}\label{nneq1}
		\lim_{n\to\infty}\mathcal{E}_{\alpha}\left(e^{i(x_n^j,t_n^j)(\cdot,|\cdot|^\alpha)}f_n\right)(x,t)=\mathcal{E}_{\alpha}\phi^j(x,t), \quad \quad \text{a.e.} \;\; (x,t) \in \bR^{d+1}.
	\end{equation} 
	For any $p_1\in(1,2]$ and $q_1=\frac{d+\alpha}{d}p_1'$, the size and support conditions on $f_n$ imply that $\{f_n\}\subset L^{p_1}(\mathbb{R}^d)$. Then by the boundness of $\mathcal{E}_{\alpha}$, we have
	\[\left\{\mathcal{E}_{\alpha}\left(e^{i(x_n^j,t_n^j)(\cdot,|\cdot|^\alpha)}f_n\right)(x,t)\right\}\subset L^{q_1}(\mathbb{R}^{d+1}).\]
	Combining this with pointwise estimate \eqref{nneq1}, the classical Br\'ezis-Lieb lemma \cite[Section 1.9]{LL2001} implies the following $L^{q_1}$-orthogonality conclusion
	\begin{equation}\label{nneq2}
		\left\Vert \mathcal{E}_{\alpha} f_n\right\Vert_{L^{q_1}(\mathbb{R}^{d+1})}^{q_1} = \sum_{j=1}^J \left\Vert \mathcal{E}_{\alpha}\phi^j \right\Vert_{L^{q_1} (\mathbb{R}^{d+1})}^{q_1} + \left\Vert \mathcal{E}_{\alpha} \omega_n^J \right\Vert_{L^{q_1} (\mathbb{R}^{d+1})}^{q_1} +o_{n\to\infty}(1),
	\end{equation}
	which includes the desired orthogonality \eqref{eq28}. 
	Choosing $p_1\in(1,p)$, the $L^{q_1}$-orthogonality conclusion \eqref{nneq2} also implies that up to subsequences $\Vert \mathcal{E}_{\alpha} \omega_n^J \Vert_{L^{q_1}}$ is uniformly bounded for all $n$ and $J$. Combining this fact with the estimate \eqref{eq14}, one can use Hölder's inequality to deduce the desired property \eqref{eq29}.
		
	This leaves us to prove the quasi-orthogonality inequality \eqref{eq27} for fixed $J\in\mathbb{N}$ with $J\leq J_0$. We choose compactly supported smooth nonnegative functions $(\varphi,\psi)$ which satisfy
	\[\|\varphi\|_{L^{\infty}(\bR^d)}=\Vert \psi\Vert_{L^1(\mathbb{R}^d)}=1, \quad \Vert\psi*(\varphi\phi^j)-\phi^j\Vert_{L^p(\mathbb{R}^d)}<\varepsilon.\]
	Recalling the projection operator $\pi_n^j$ in Lemma \ref{17} defined as follows
	\[\pi(x_n^j,t_n^j)f(\xi) :=e^{-i(x_n^j,t_n^j)(\xi,|\xi|^{\alpha})} \left(\psi* \l(e^{i(x_n^j,t_n^j)(\cdot,|\cdot|^{\alpha})} \varphi f\r)\right)(\xi),\]
	and by the previous decomposition \eqref{E:Profile decomposition-Lp-1}, we have
	\begin{align*}
		\pi_n^j f_n(\xi) &=\sum_{k:k\ne j,k\leq J}e^{-i(t_n^j,x_n^j)(\xi,|\xi|^{\alpha})}\left(\psi*\left(\varphi e^{i(t_n^j-t_n^k,x_n^j-x_n^k)(\cdot,|\cdot|^{\alpha})}\phi^k\right)\right)(\xi) \\
		&+e^{-i(t_n^j,x_n^j)(\xi,|\xi|^{\alpha})}\left(\psi*(\varphi\phi^j)\right)(\xi) +e^{-i(t_n^j,x_n^j)(\xi,|\xi|^{\alpha})}\left(\psi*(\varphi e^{i(t_n^j,x_n^j)(\cdot,|\cdot|^{\alpha})}\omega_n^{J})\right)(\xi).
	\end{align*}
	Next we estimate each term. First, the property \eqref{eq26} implies that
	\[e^{-i(t_n^j-t_n^{k},x_n^j-x_n^{k})(\cdot,|\cdot|^{\alpha})}\phi^k\rightharpoonup0 \quad \text{in} \;\; {L^p(\mathbb{R}^d)} \quad \text{as} \;\; n\to\infty,\quad \forall k\ne j.\]
	By the boundedness and compact supports of $\psi$ and $\varphi$, the dominated convergence theorem further implies that for arbitrary fixed $j\neq k$ there holds
	\begin{equation} \label{E:Profile decomposition-Lp-2}
		\lim_{n\to\infty} \left\Vert \pi_n^j \l(e^{-i(t_n^k,x_n^k) (\cdot,|\cdot|^{\alpha})} \phi^k\r) \right\Vert_{{L^p(\mathbb{R}^d)}} =\lim_{n\to\infty} \left\Vert \psi*\left(\varphi e^{i(t_n^j-t_n^k,x_n^j-x_n^k) (\cdot,|\cdot|^{\alpha})} \phi^k \right) \right\Vert_{{L^p(\mathbb{R}^d)}} =0.
	\end{equation}
	And by \eqref{eq30} one can similarly obtain
	\[\lim_{n\to\infty} \left\Vert \pi_n^j \l(\omega_n^J \r) \right\Vert_{L^p(\mathbb{R}^d)} = \lim_{n\to\infty}\left\Vert\psi*\left(\varphi e^{i(t_n^j,x_n^j)(\cdot,|\cdot|^{\alpha})}\omega_n^{J}\right)\right\Vert_{L^p(\mathbb{R}^d)}=0.\]
	In summary, these estimates and the choice of $(\varphi,\psi)$ show that for each $j\leq J$ there holds 
	\[\lim_{n\to\infty}\left\Vert\pi_n^jf_n\right\Vert_{L^p(\mathbb{R}^d)}=\Vert\psi*(\varphi\phi^j)\Vert_{L^p(\bR^d)} = \|\phi^j\|_{L^p(\mathbb{R}^d)}+ o_{\varepsilon\to 0}(1).\]
	Then by Lemma \ref{17}, for fixed $J$, we conclude
	\begin{align*}
		\left\Vert \Vert \phi^j \Vert_{L^p} \right\Vert_{\ell^{p^*}} + o_{\varepsilon\to0}(1) &= \lim_{n\to\infty} \left\Vert \Vert \pi_n^j f_n \Vert_{L^p} \right\Vert_{\ell^{p^*}} = \lim_{n\to\infty} \l\Vert \Pi_n^J f_n \r\Vert_{\ell^{p^*}(L^p)} \leq \limsup_{n\to\infty} \Vert f_n \Vert_{L^p}.
	\end{align*}
	Finally, letting $\varepsilon\to0$, we obtain the desired conclusion \eqref{eq27} and finish the proof.
\end{proof}

\subsection{The $L^p$ convergence} \label{SubS:The Lp convergence}
Here we prove our main result Theorem \ref{n1} by using the frequency localization and the $L^p$-based profile decomposition obtained in the previous two subsections.
\begin{proof}[\textbf{Proof of Theorem \ref{n1}}]
	Let $\{f_n\}$ be an extremal sequence for $M_{\alpha,p}$. By Proposition \ref{18}, up to scaling there holds
	$$\lim_{m\to\infty}\limsup_{n\to\infty}\left\Vert f_n\chi_{\left(\{\xi:|\xi|>m\}\cup\left\{\xi:\left|f_n(\xi)\right|>m\right\}\right)}\right\Vert_{L^p(\mathbb{R}^{d})}=0,$$
	which gives the following
	\begin{equation*} 
		\lim_{m\to\infty}\limsup_{n\to\infty} \left\Vert \mathcal{E}_{\alpha} \left(f_n\chi_{\left(\{\xi:|\xi|>m\}\cup \left\{\xi:\left|f_n(\xi)\right|>m\right\} \right)} \right) \right\Vert_{L^q(\mathbb{R}^{d+1})}=0.
	\end{equation*}
	For convenience, we denote
	\[f_n^m :=f_n\chi_{\{\xi: |\xi|<m \} \cap\{\xi:|f_n(\xi)|<m\}},\]
	and then there holds
	\[\lim_{m\to\infty}\limsup_{n\to\infty} \| f_n^m \|_{L^p(\mathbb{R}^d)} = 1,\quad \lim_{m\to\infty}\limsup_{n\to\infty} \|\mathcal{E}_{\alpha} f_n^m \|_{L^q(\mathbb{R}^{d+1})} = M_{\alpha,p}.\]
	By Proposition \ref{19}, up to subsequences in $n$ (which is independent of $m$) we decompose $f_n^m$ as
	\[f_n^m=\sum\limits_{j=1}^J e^{-i(x_n^{m,j},t_n^{m,j})(\cdot,|\cdot|^\alpha)}\phi^{m,j}+\omega_n^{m,J},\quad 1\leq J\leq J_0, \quad J_0\in\mathbb{N}_+\cup\{\infty\}.\]
	Then using the properties \eqref{eq27}-\eqref{eq29} and $L^{q_1}$-orthogonality conclusion, noticing that $q>p'$, for each $m$ we conclude that
	\begin{align*}
		M_{\alpha,p}^q &-\limsup_{n\to\infty} \Vert \mathcal{E}_{\alpha}\left(f_n\chi_{\{\xi:|\xi|>m\}\cup\left\{\xi:\left|f_n(\xi)\right|>m\right\}}\right) \Vert_{L^q(\mathbb{R}^{d+1})}^q \\
		&\leq \limsup_{n\to\infty} \Vert \mathcal{E}_{\alpha}f_n^m \Vert_{L^q(\mathbb{R}^{d+1})}^q =\sum_{j=1}^{J_0} \Vert \mathcal{E}_{\alpha}\phi^{m,j} \Vert_{L^q(\mathbb{R}^{d+1})}^q \notag\\
		&\leq M_{\alpha,p}^{p^*} \left(\max_{1\leq j\leq J_0} \Vert \mathcal{E}_{\alpha}\phi^{m,j} \Vert_{L^q(\mathbb{R}^{d+1})} \right)^{q-p'} \sum_{j=1}^{J_0} \Vert\phi^{m,j} \Vert_{L^p(\mathbb{R}^d)}^{p'} \\
		&\leq M_{\alpha,p}^{p^*}\left(\max_{1\leq j\leq J_0} \Vert\mathcal{E}_{\alpha}\phi^{m,j}\Vert_{L^q(\mathbb{R}^{d+1})}\right)^{q-p'}.
	\end{align*}
	For each $m$, since $\sum_j \Vert\mathcal{E}_{\alpha}\phi^{m,j}\Vert_{L^q}$ is bounded, we can choose $j=j_m$ to maximize $\Vert\mathcal{E}_{\alpha}\phi^{m,j}\Vert_{L^q}$ and define
	\[\phi^m :=\phi^{m,j_m}, \quad (x_n^m,t_n^m) :=(x_n^{m,j_m}, t_n^{m,j_m}).\]
	Thus we conclude
	$$M_{\alpha,p}^q -o_{m\to\infty} (1) \leq M_{\alpha,p}^{p'}\Vert\mathcal{E}_{\alpha}\phi^{m}\Vert_{L^q(\mathbb{R}^{d+1})}^{q-p'}\leq M_{\alpha,p}^q\Vert\phi^m\Vert_{L^p(\mathbb{R}^d)}^{q-p'}\leq M_{\alpha,p}^q.$$
	Letting $m\to\infty$, this inequality and the property \eqref{eq27} imply that 
	$$\lim_{m\to\infty}\Vert \phi^m\Vert_{L^p(\mathbb{R}^d)}=1,\quad \lim_{m\to\infty}\Vert\mathcal{E}_{\alpha}\phi^m\Vert_{L^p(\mathbb{R}^d)}=M_{\alpha,p}.$$
	This means that $\phi^m$ is a quasi-extremal sequence in the sense that its $L^p$-norm has limit one (but not identically one), and moreover, the property \eqref{eq30} implies that
	\[e^{i(x_n^mt_n^m)(\cdot,|\cdot|^\alpha)}f_n^m\rightharpoonup\phi^m,\quad \text{in} \;\;{L^p(\mathbb{R}^d)} \quad \text{as} \;\; n\to\infty.\]
	Then for each fixed $m$, the lower semicontinuity of norms \cite[Section 2.11]{LL2001} implies
	\begin{align*}
		\|2\phi^m\|_{L^p(\bR^d)} &\leq \limsup_{n\to\infty} \l\|e^{i (x_n^m,t_n^m)(\cdot,|\cdot|^{\alpha})} f_n^m + \phi^m \r\|_{L^p(\mathbb{R}^d)} \\
		&\leq \limsup_{n\to\infty} \l\|e^{i (x_n^m,t_n^m)(\cdot,|\cdot|^{\alpha})} f_n^m \r\|_{L^p(\mathbb{R}^d)} + \|\phi^m\|_{L^p(\mathbb{R}^d)} \\
		&\leq 1+ \|\phi^m\|_{L^p(\mathbb{R}^d)},
	\end{align*}
	which further gives
	\begin{equation} \label{E:Precompactness-0.8}
		\lim_{m\to\infty} \limsup_{n\to\infty} \l\|e^{i (x_n^m,t_n^m)(\cdot,|\cdot|^{\alpha})} f_n^m + \phi^m \r\|_{L^p(\mathbb{R}^d)}=2.
	\end{equation}
	Here we introduce the following notations
	\[A(m,n):= \l\Vert e^{i(x_n^mt_n^m)(\cdot,|\cdot|^\alpha)} f_n^m +\phi^m \r\Vert_{L^p(\mathbb{R}^d)}, \quad B(m,n):= \l\Vert e^{i(x_n^mt_n^m)(\cdot,|\cdot|^\alpha)} f_n^m - \phi^m \r\Vert_{L^p(\mathbb{R}^d)}.\]
	The uniform convexity of $L^p$ norms \cite[Section 2.5]{LL2001} implies
	\begin{align*}
		\l[ A(m,n) + B(m,n) \r]^p + \l|A(m,n)- B(m,n) \r|^p \leq 2^p \l(\Vert\phi^m \Vert_{L^p(\mathbb{R}^d)}^p + \Vert f_n^m\Vert_{L^p(\mathbb{R}^d)}^p\r) \leq 2^{p+1},
	\end{align*}
	which, recalling \eqref{E:Precompactness-0.8} and the fact $F(x)=|x+C|^p$ is strictly convex for all $p>1$, further gives the following result \eqref{E:Precompactness-1} by letting $n\to\infty$ and then $m\to\infty$, 
	\begin{equation} \label{E:Precompactness-1}
		\lim\limits_{m\to\infty} \limsup\limits_{n\to\infty} \left\Vert f_n^m- e^{-i(x_n^m,t_n^m)(\cdot,|\cdot|^\alpha)} \phi^m \right\Vert_{L^p(\mathbb{R}^d)}=0.
	\end{equation}
	By this consequence, we can use the Proposition \ref{18} and Minkowski's inequality to deduce
	\begin{equation} \label{E:Precompactness-2}
		\lim\limits_{m\to\infty}\limsup\limits_{n\to\infty}\left\Vert f_n- e^{-i(x_n^m,t_n^m)(\cdot,|\cdot|^\alpha)} \phi^m \right\Vert_{L^p(\mathbb{R}^d)}=0,
	\end{equation}
	which further deduces
	\[\limsup\limits_{n\to\infty} \left\Vert e^{-i(x_n^{\widetilde{m}},t_n^{\widetilde{m}}) (\cdot,|\cdot|^\alpha)} \phi^{\widetilde{m}}- e^{-i(x_n^m,t_n^m) (\cdot,|\cdot|^\alpha)} \phi^m \right\Vert_{L^p(\mathbb{R}^d)}=o_{\min\{m,\widetilde{m}\}\to \infty}(1).\]
	Applying the projection operator $\pi_n^m$, recalling the estimate \eqref{E:Profile decomposition-Lp-2} and that $\phi^m$ is a quasi-extremal sequence, for sufficiently large $m$ and $\widetilde{m}$ one can conclude that $|(x_n^m -x_n^{\widetilde{m}}, t_n^m - t_n^{\widetilde{m}})| <\infty$ uniform in $n$. Hence for the relation \eqref{E:Precompactness-2}, by applying space-time translations to $f_n$, we may assume that $(x_n^M,t_n^M) \equiv 0$ for some fixed $M$ large enough; then for each $m\geq M$, up to subsequences, we may further assume $(x_n^m,t_n^m) \to (x^m, t^m)$ as $n$ goes to infinity. These assumptions and the relation \eqref{E:Precompactness-2}, by Minkowski's inequality, implies the following fact
	\[\lim\limits_{m\to\infty}\limsup\limits_{n\to\infty}\left\Vert f_n- \widetilde{\phi}^m \right\Vert_{L^p(\mathbb{R}^d)}=0, \quad \widetilde{\phi}^m := e^{-i(x^m,t^m) (\cdot,|\cdot|^\alpha)} \phi^m.\]
	By this fact and that $\widetilde{\phi}^m$ is a quasi-extremal sequence (since $\phi^m$ is a quasi-extremal sequence), we can use Minkowski's inequality to conclude that $\widetilde{\phi}^m$ is a Cauchy sequence which converges to some nonzero function $\phi\in {L^p(\mathbb{R}^d)}$; and further conclude that $f_n$ converges to this function $\phi$ as well. Therefore, we obtain that $\{f_n\} \subset {L^p(\mathbb{R}^d)}$ is a compact sequence and $\phi$ is an extremal function. In summary, we obtain the desired precompactness up to symmetries and complete the proof.
\end{proof}

\bigskip\bigskip
\begin{flushleft}
	\vspace{0.3cm}\textsc{Boning Di \\
		School of Mathematical Sciences, University of Chinese Academy of Sciences, Beijing, 100049, People's Republic of China \\
		Academy of Mathematics and Systems Science, Chinese Academy of Sciences, Beijing, 100190, People's Republic of China} \\
	\emph{E-mail address}: \textsf{diboning@amss.ac.cn}
	
	\vspace{0.3cm}\textsc{Ning Liu \\
		School of Mathematical Sciences, University of Chinese Academy of Sciences, Beijing, 100049, People's Republic of China} \\
	\emph{E-mail address}: \textsf{liuning21@mails.ucas.ac.cn}
	
	\vspace{0.3cm}\textsc{Dunyan Yan \\
		School of Mathematical Sciences, University of Chinese Academy of Sciences, Beijing, 100049, People's Republic of China} \\
	\emph{E-mail address}: \textsf{ydunyan@ucas.ac.cn}
\end{flushleft}

\end{document}